\documentclass[11pt]{article}
\usepackage{amssymb}
\usepackage[centering,height=220mm,width=150mm]{geometry}
\usepackage{tikz}
\usepackage{xcolor}
\usepackage{mathrsfs}
\usepackage{amsfonts}
\usepackage{amsmath}
\numberwithin{equation}{section}
\allowdisplaybreaks[4]
\newtheorem{Theorem}{Theorem}[section]
\newtheorem{Proposition}[Theorem]{Proposition}
\newtheorem{Corollary}[Theorem]{Corollary}
\newtheorem{Lemma}[Theorem]{Lemma}
\newtheorem{Example}[Theorem]{Example}
\newtheorem{Remark}[Theorem]{Remark}
\newtheorem{Definition}[Theorem]{Definition}

\def\bbZ{\mathbb{Z}}
\def\bbR{\mathbb{R}}

\begin{document}

\title{Iterated and Mixed Weak Norms with Applications to Geometric Inequalities\thanks{This work was partially supported by the
National Natural Science Foundation of China (11525104, 11531013 and 11761131002).}}

\author{Ting Chen \quad and  \quad  Wenchang Sun\\
School of Mathematical Sciences and LPMC,  Nankai University,
      Tianjin~300071, China \\
   Emails: t.chen@nankai.edu.cn, sunwch@nankai.edu.cn}

\date{}
\maketitle

\begin{abstract}
In this paper, we consider a new weak norm, iterated weak norm
in Lebesgue spaces with mixed norms.
We study properties of the mixed weak norm and  the iterated weak norm
and present the relationship between the two   weak norms.
Even for the ordinary Lebesgue spaces, the two  weak norms are not equivalent and any one
of them can not control the other one.
We give some convergence and completeness results for the two weak norms respectively.
We study the convergence in truncated norm, which is a substitution of the convergence in measure for mixed Lebesgue spaces.
And we give a characterization of the convergence in truncated norm.
We show that H\"older's inequality is not always true on mixed weak spaces
and we give a complete characterization of indices which admit   H\"older's inequality.
As applications, we  establish some geometric inequalities related to fractional integrals  in mixed weak spaces and in iterated weak spaces respectively,
which essentially generalize the Hardy-Littlewood-Sobolev inequality.
\end{abstract}



\section{Introduction}

For $\vec p =(p_1, \ldots, p_r)$ and a measurable function $f$ defined on $\bbR^{n_1}\times\ldots\times \bbR^{n_r}$,
where $p_i$ are positive numbers and  $n_i$ are positive integers, $1\le i\le r$,
we define the  $L^{\vec p}$ norm of $f$ by
\[
  \|f\|_{L^{\vec p}} :=   \Big\|  \| f\|_{L_{x_1}^{p_1}} \cdots \Big\|_{L_{x_r}^{p_r}}.
\]
The Lebesgue space $L^{\vec p}(\bbR^{n_1}\times\ldots\times \bbR^{n_r})$ with mixed norms consists of all measurable functions $f$
for which $\|f\|_{L^{\vec p}}<\infty$.
For $\vec p =(p_1, p_2)$, we also write $L^{\vec p}$ as $L^{p_2}(L^{p_1})$.

Lebesgue spaces with mixed norms
were first studied by Benedek and Panzone in  \cite{Benedek1961}, where
many fundamental properties were proved. In particular, they showed that
such spaces possess similar  properties as usual Lebesgue spaces.
See also related works by Benedek, Calder\'on  and  Panzone \cite{Benedek1962},
Rubio de Francia, Ruiz  and  Torrea \cite{Rubio1986},
and Fernandez \cite{Fernandez1987}.

Recently, many works have been done for Lebesgue spaces with mixed norms.
For example, Kurtz \cite{Kurtz2007} proved    that some classical operators, which include the strong maximal
function, the double Hilbert transform and singular integral operators, are bounded on weighted
Lebesgue spaces with mixed norms.
Torres  and Ward \cite{Torres2015} gave Calder\'on's reproducing formula and wavelet characterization
of such spaces.
In recent works by
Cleanthous, Georgiadis, and Nielsen~\cite{Cleanthous2017} and
Huang, Liu,  Yang and Yuan \cite{HuangLiuYangYuan2018}
anisotropic mixed-norm
Hardy spaces were also studied.
In this paper, we focus on weak norms.

Recall that the Hardy-Littlewood-Sobolev inequality says that for any $f\in L^{p_1}(\bbR^n)$
and $g\in L^{p_2}(\bbR^n)$, where $1< p_1, p_2< \infty$ with  $1/p_1 +1/p_2   > 1$,
we have
\begin{equation}\label{eq:Hardy-Littlewood-Sobolev}
\int_{\mathbb{R}^{n}}  \int_{\mathbb{R}^{n}}  \frac{ f(x) g(y)}{ |x-y|^{n(2-1/p_1 -1/p_2  )}} dxdy \leq C_{\vec p,n}  \|f\|_{L^{p_1}} \|g\|_{L^{p_2}}.
\end{equation}
See the works by Beckner \cite{Beckner1995,Beckner2008a,Beckner2008b}, Burchard \cite{Burchard1996}, Carlen and Loss \cite{Carlen1990}, Lieb \cite{Lieb1983},  Wu, Shi and Yan~\cite{Wu2014} and
see also Lieb's \cite{Lieb2001} and Stein's \cite{Stein1970} books for the Hardy-Littlewoood-Sobolev inequality and the sharp versions.
The multilinear analogues of the Hardy-Littlewood-Sobolev inequality were studied
by Beckner \cite{Beckner1995}, Gressman \cite{Gressman2011}, and Valdimarsson \cite{Valdimarsson2012}.
Besides, we refer to Christ \cite{Christ1984,Christ1985}, Dury \cite{Drury1988,Drury1985,Drury1987},
and Gressman \cite{Gressman2007}, Tao and Wright \cite{Tao2003}
for some related works regarding the $k$-plane transform and the restriction of the Fourier transform.

Observe that $\|f\|_{L^{p_1}} \|g\|_{L^{p_2}} = \|f\otimes g\|_{L^{\vec p}}$,
where $f\otimes g(x,y) := f(x)g(y)$ and $\vec p = (p_1, p_2 )$.
Define
\[
      L_{\gamma} f(x,y) =  \frac{f(x,y)}{|x-y|^{\gamma} },   \ \gamma>0.
\]
For $\gamma = n(2-1/p_1 -1/p_2  )$, (\ref{eq:Hardy-Littlewood-Sobolev}) says that
\[
    \|L_{\gamma} f\otimes g \|_{L^{\vec 1}} \lesssim \|f\otimes g\|_{L^{\vec p}}, \qquad f\in L^{p_1}(\bbR^n), \,\, g\in L^{p_2}(\bbR^n).
\]
It is natural to ask if the above inequality is still true whenever $f\otimes g$ is replaced by a general function in $L^{\vec p}(\bbR^n\times \bbR^n)$?
More precisely, do we have
\[
   \|L_{\gamma} f  \|_{L^{\vec q}} \lesssim \|f\|_{L^{\vec p}},  \qquad \forall f\in L^{\vec p}(\bbR^n\times \bbR^n)
\]
for appropriate $\vec p, \vec q$ and $\gamma$?

The answer is false in general.
Fortunately, the above inequality is true if the $L^{\vec p}$
and $L^{\vec q}$ norms are replaced with some other mixed norms, respectively. For example,
for appropriate indices, we have
\[
  \|L_{\gamma} f\|_{L^{p_2}(L^{r,\infty})}  \leq C_{\vec p,r,n} \|f\|_{L^{p_2}(L^{p_1,\infty})},
\]
where $\|f\|_{L^{p_2}(L^{p_1,\infty})}:= \| \|f(x,y)\|_{L_x^{p_1,\infty}}\|_{L_y^{p_2} }$.
For a complete characterization of  $L_{\gamma}$ with respect to various choices of indices and mixed norms, we refer to
Theorem~\ref{thm:weak strong norm fractional}.

Next we consider another variant of   (\ref{eq:Hardy-Littlewood-Sobolev}).

By replacing $g$ with $g(-\cdot)$ and a change of variable, we get
\[
  \int_{\mathbb{R}^{n}}  \int_{\mathbb{R}^{n}}  \frac{ f(x) g(y)}{ |x+y|^{n(2-1/p_1 -1/p_2  )}} dxdy \leq C_{\vec p,n}  \|f\|_{L^{p_1}} \|g\|_{L^{p_2}}.
\]
Observe that
\[
  \frac{1}{  (|x+y|+|x-y|)^{\gamma}} \le  \frac{1}{   |x+y|^{\gamma}}
  +  \frac{1}{  |x-y|^{\gamma}}.
\]
This prompts us to consider the following operator
\[
      T_{\gamma} f(x,y) = \frac{f(x,y)}{  (|x+y|+|x-y|)^{\gamma} },  \qquad  \gamma>0.
\]
We see from   (\ref{eq:Hardy-Littlewood-Sobolev}) that for $\gamma = n(2-1/p_1 -1/p_2  )$,
\[
  \|T_{\gamma}f\otimes g\|_{L^{\vec 1}} \lesssim  \| f\otimes g\|_{L^{\vec p}}.
\]
We ask if the following inequality
\[
  \|T_{\gamma} f\|_{L^{\vec q}} \lesssim  \| f \|_{L^{\vec p}},\qquad  \forall f \in L^{\vec p}
\]
is true for some $\vec p$ and $\vec q$?

The answer is again negative.
In fact, since $(|x+y|+|x-y|)^{-\gamma}\not\in L^{\vec r}$ for any $\vec r$ with $0<r_1, r_2\le \infty$,
 the above inequality is false.
Moreover, the following  inequality
\begin{equation}\label{eq: T weak a}
\|T_{\gamma} f\|_{L^{\vec q,\infty}}  \leq C_{\vec p, \vec q, n,\gamma} \|f\|_{L^{\vec p, \infty}}
\end{equation}
or
\begin{equation}\label{eq: T weak b}
\|T_{\gamma} f\|_{L^{\vec q,\infty} }      \leq C_{\vec p, \vec q, n,\gamma} \|f\|_{L^{\vec p} }
\end{equation}
is also false in general, where
\[
\|f\|_{L^{\vec p,\infty}} :=  \sup_{\lambda>0}     \lambda \|  \chi^{}_{ \{ |f|>\lambda \} } \|_{L^{\vec p}}
\]
is the mixed weak $L^{\vec p}$  norm of $f$.

When the mixed weak norm is  replaced by the iterated weak norm defined by
\[
  \|f\|_{L^{(p_2,\infty)}( L^{(p_1,\infty)})} :=   \Big\|  \| f(x,y)\|_{L_{x}^{p_1,\infty}}  \Big\|_{L_{y}^{p_2,\infty}},
\]
we get a positive conclusion. Specifically,
 for all $0<q_1 \leq p_1\leq\infty$ and $0<q_2 \leq p_2\leq \infty$ satisfying the homogeneity condition
$$\frac{1}{q_1}+\frac{1}{q_2}=  \frac{1}{p_1}+\frac{1}{p_2}  +\frac{\gamma}{n},$$
we have
\begin{equation}\label{eq:fractional a}
\|T_{\gamma} f\|_{L^{q_2,\infty}(L^{q_1,\infty})}  \leq C_{\vec p,\vec q, n} \|f\|_{L^{p_2,\infty}(L^{p_1,\infty})}.
\end{equation}

We show in Corollary~\ref{Co:Chen}
 that for $\vec p=(\infty, \infty)$, (\ref{eq:fractional a}) implies the following geometric inequality studied in
\cite{Chen2017},
\begin{equation}\label{eq: sup ineq}
  \|f\|_{q_1}\|g\|_{q_2} \lesssim \sup |f(x)g(y)| \cdot |x-y|^{n/q_1+n/q_2}.
\end{equation}
When $f=g=\chi_{E}$, $E\subset \bbR^n$, (\ref{eq: sup ineq}) becomes
$$|E|^{1/n} \lesssim \sup\limits_{x,y \in E}  |x-y|,$$
which is a well known geometric extremal problem named isodiametric inequality:
amongst all sets with given diameter the ball has the maximal volume.

We show in Section 3 that the inverse version of (\ref{eq:fractional a})
 implies the forward and inverse
Hardy-Littlewood-Sobolev inequalities. Therefore, it is a  generalization of the
Hardy-Littlewood-Sobolev inequalities. Moreover, our approach gives a new method to prove the boundedness
of the fractional integral
\[
   I_{\alpha} f(x) := \int_{\bbR^n} \frac{f(y)}{|x-y|^{n-\alpha}} dy
\]
from $L^{p_1}$ to $L^{q_1}$, where $1/q_1=1/p_1 - \alpha/n$.

Since (\ref{eq:fractional a}) is true and (\ref{eq: T weak a}) is false in general, it is interesting to investigate these two  weak norms.
We show in Section 2 that the two weak norms are not equivalent even if $p_1=p_2$.
Moreover,  one of  them can not control the other one.

The paper is organized as follows. In Section 2 we study
various aspects of the two weak norms. In particular, we give the relationship between the two weak norms.
We give some convergence and completeness results for both weak norms respectively.
We study the convergence in truncated norm, which is a substitution of the convergence in measure for mixed Lebesgue spaces.
We give a characterization of  the convergence in truncated norm.
We prove that H\"older's inequality holds for iterated weak spaces but it might be false in mixed weak spaces.
We give a complete characterization of indices for which H\"older's inequality
is true in mixed weak spaces.
Besides, we give some interpolation properties in their respective spaces.
In Section 3 we establish some geometric inequalities.

Throughout the paper,
$A  \lesssim B$ means that $A \leq C B$,
and the letter $C$ stands for positive constants
that are not necessarily the same at each occurrence but that are independent of the essential variables.
$A \gtrsim B$ and $A\approx B$ are defined similarly.

\section{Weak Norms}

In this section, we study various aspects of weak norms, which include the relationship between the two weak norms,
H\"older's inequality, the interpolation formula, and the convergence of sequences of functions in weak norms.

\subsection{Comparison between two weak norms}

For simplicity, we consider only the case of $\vec p = (p_1, p_2)$.
In this case, the iterated weak norm on $\bbR^n\times \bbR^m$
is
\begin{align*}
  \|f\|_{L^{p_2,\infty}(L^{p_1,\infty})} &=   \Big\|  \| f\|_{L_{x}^{p_1,\infty}}  \Big\|_{L_{y}^{p_2,\infty}} \\
  & = \sup_{\gamma>0} \gamma \left|\left\{y:\, \sup_{\lambda>0}  \lambda | \{x:\, |f(x,y)|>\lambda\}|^{1/p_1} > \gamma  \right\}\right|^{1/p_2}\\
  & = \sup_{\gamma>0} \gamma \left|\left\{y:\, \sup_{\lambda>0}  \lambda | E_{y,\lambda}|^{1/p_1} > \gamma  \right\}\right|^{1/p_2},
\end{align*}
where
\begin{equation}\label{eq:E y}
E_{y,\lambda} := \{x:\, |f(x,y)|> \lambda\}.
\end{equation}
And the mixed weak norm is
\begin{align*}
\|f\|_{L^{\vec p,\infty}} &=  \sup_{\lambda>0}     \lambda \|  \chi^{}_{ \{ |f|>\lambda \} } \|_{L^{\vec p}}
 =  \sup_{\lambda>0}     \lambda  \left( \int |E_{y,\lambda }|^{p_2/p_1} \mathrm{d}y\right)^{1/p_2}.
\end{align*}

It is well known that for $q > 0$,
\begin{equation}\label{eq:weak:1}
\|f\|_{L^{q,\infty}} \le \|f\|_{L^{q}}.
\end{equation}
Hence
\[
  \|f\|_{L^{p_2,\infty}(L^{p_1,\infty})} \le \|f\|_{L^{\vec p}}.
\]

Next we compare the two weak norms.
The following lemma could be known. Since we do not find a proof, we include a proof here.
\begin{Lemma}\label{Lm:weak}
Suppose that $0<  q<\infty$ and $f\in L^q(\bbR^n)$. Then  the equality in (\ref{eq:weak:1}) holds if and only if $|f| = C \chi_E$ for some constant $C$
and $E\subset \bbR^n$.
\end{Lemma}

\begin{proof}
Assume that $|f|$ is not of the form $C\chi_E$.
Then there exist positive numbers $a<b$ such that
$|\{0<|f|<a\}|>0$
and $|\{|f|>b\}|>0$.

For $0<\lambda<a$, we have
\begin{align*}
\lambda^q |\{|f|>\lambda \}|
&= \int_{\{|f|>\lambda \}} \lambda^q \\
&\le \int_{\{\lambda < |f| <b\}} |f|^q
      + \int_{\{|f| \ge b\}} (|f|^q - b^q + a^q) \\
&\le \|f\|_q^q - (b^q - a^q)|\{|f|\ge b\}|.
\end{align*}
And for $\lambda \ge a$, we have
\begin{align*}
\lambda^q |\{|f|>\lambda \}|
&= \int_{\{|f|>\lambda \}} \lambda^q
\le \|f\|_q^q - \|f\cdot \chi^{}_{\{0<|f|<a\}}\|_q^q.
\end{align*}
Hence
\[
  \|f\|_{L^{q,\infty}}
  = \sup_{\lambda>0} \lambda |\{|f|>\lambda \}|^{1/q}
  < \|f\|_{L^{q}}.
\]
This completes the proof.
\end{proof}

Next we  illustrate that iterated weak norms are order dependent.
Moreover, the mixed weak norm and the iterated weak norm are not equivalent and any one of them can not control the other one.
Specifically, we have the following.
\begin{Theorem} \label{thm::weak norms}
Suppose that $0<p_1, p_2 < \infty$ and $m$ and $n$ are positive integers. We have

\begin{enumerate}
\item
   $L^{p_2, \infty}(L^{p_1,\infty})(\bbR^n\times \bbR^m)
   \not\subset L^{\vec p, \infty}(\bbR^n\times \bbR^m)$
   and
   $L^{\vec p, \infty}(\bbR^n\times \bbR^m)
    \not\subset
   L^{p_2, \infty}(L^{p_1,\infty})(\bbR^n\times \bbR^m)$.

\item
$L^{p_1,\infty}_x(L^{p_2, \infty}_y)  \not\subset   L^{p_2, \infty}_y(L^{p_1,\infty}_x) $
and $ L^{p_2, \infty}_y(L^{p_1,\infty}_x)  \not\subset  L^{p_1,\infty}_x(L^{p_2, \infty}_y)  $.

\item
      $L^{\vec p}
      \subsetneq L^{\vec p, \infty}
    \bigcap L^{p_2, \infty}(L^{p_1,\infty}) $.

\end{enumerate}

\end{Theorem}

\begin{proof}
(i).\,\, First, we consider the function $F(x,y) =1/( |x|^{n/p_1} |y|^{m/p_2})$.
It is easy to see that
\[
  \|F|_{L^{p_2, \infty}(L^{p_1,\infty})} = v_n^{1/p_1} v_m^{1/p_2},
\]
where $v_n$ and $v_m$ are  the  volumes of  unit balls in $\bbR^n$ and $\bbR^m$, respectively.

On the other hand, for any $\lambda>0$,
\begin{align*}
  \big\|  \chi^{}_{\{F>\lambda\}} \big\|_{L^{\vec p}}
  &= \left(\int_{|y|>0}   \frac{v_n^{p_2/p_1}}{\lambda^{p_2}|y|^{m}} dy \right)^{1/p_2} = \infty.
\end{align*}
Hence  $\|F\|_{L^{\vec p, \infty}} = \infty$. This proves $L^{p_2, \infty}(L^{p_1,\infty})
   \not\subset L^{\vec p, \infty}$.

Next we consider the function $G(x,y) = a^{ |y|^{m}} \chi^{}_{[0, a^{-p_1|y|^{m}/n}]}(|x|)$, where $a>1$ is a constant.
For any  $\lambda>0$, we have
\[
  |\{x:\, G(x,y)>\lambda \}| =
  \begin{cases}
  v_n a^{-p_1|y|^{m}},  & \lambda < a^{ |y|^{m}}, \\
  0,    & \mathrm{others}.
  \end{cases}
\]
Hence
\[
   \| G(\cdot, y)\|_{L^{p_1,\infty}} = v_n^{1/p_1},  \qquad |y|>0.
\]
Therefore,
\[
    \|G\|_{L^{p_2, \infty}(L^{p_1,\infty})} = \infty.
\]

On the other hand,
\begin{align*}
\bigg(\int_{\bbR^m}|\{x:\, G(x,y)>\lambda \}|^{p_2/p_1} dy\bigg)^{1/p_2}
&= \bigg(\int_{a^{|y|^{m}}>\lambda}  v_n^{p_2/p_1} a^{-p_2|y|^{m}} dy \bigg)^{1/p_2}\\
&= \frac{v_n^{1/p_1} v_m^{1/p_2}}{(p_2\ln a)^{1/p_2} \lambda }.
\end{align*}
Hence
\[
  \|G\|_{L^{\vec p, \infty}} =
  \frac{v_n^{1/p_1} v_m^{1/p_2}}{(p_2\ln a)^{1/p_2}  }.
\]
Therefore,
$G\in
L^{\vec p, \infty}
   \setminus
   L^{p_2, \infty}(L^{p_1,\infty})$.

(ii).\,\,
Consider the previous function $G(x,y)$ defined in (i). Fix some  $\lambda>0$ and $x\in\bbR^n$ with $|x|<1$.
We have
\[
  \{y:\, G(x,y)>\lambda \} = \left\{y:\, \frac{\ln \lambda}{\ln a} < |y|^m \le \frac{-n \ln |x|}{p_1\ln a}\right\}.
\]
Hence
\[
  |\{y:\, G(x,y)>\lambda \}| =
  \begin{cases}
  v_m \left( \frac{-n \ln |x|}{p_1\ln a} - \frac{\ln \lambda}{\ln a}  \right)  ,  & \lambda <  \frac{1}{ |x|^{n/p_1}} , \\
  0,    & \mathrm{others}.
  \end{cases}
\]
Therefore,
\[
  \sup_{\lambda>0}  \lambda   |\{y:\, G(x,y)>\lambda \}|^{1/p_2}
  = \begin{cases}
  \left(\frac{v_m}{p_2 e\ln a}\right)^{1/p_2}  \frac{1}{|x|^{n/p_1}},  &|x|<1, \\
  0, & \mathrm{others.}
  \end{cases}
\]
Consequently,
\[
   \|G\|_{L^{p_1,\infty}_x(L^{p_2, \infty}_y)} =
   \frac{v_m^{1/p_2} v_n^{1/p_1}}{(p_2 e\ln a)^{1/p_2}}
      <\infty.
\]
In other words, $G\in L^{p_1,\infty}_x(L^{p_2, \infty}_y) \setminus L^{p_2, \infty}_y(L^{p_1,\infty}_x)$.
Replacing $G(x,y)$ with $G(y,x)$  we get $L^{p_2, \infty}_y(L^{p_1,\infty}_x)\setminus  L^{p_1,\infty}_x(L^{p_2, \infty}_y) \ne \emptyset$.

(iii).\,\,
It is obvious that the left-hand side of (iii)
is a subset of the right-hand side. So we only need to show that it is a proper subset.

Take some $w\in L^{p_2}(\bbR^m)\setminus \{0\}$. Let $F(x,y) = w(y)/|x|^{n/p_1}$.
Then we have $F\not\in L^{\vec p}$.

For any $\lambda>0$, we have
\[
  |\{x:\, |F(x,y)|>\lambda \}| = v_n \frac{|w(y)|^{p_1}}{\lambda^{p_1}}.
\]
Hence
\begin{align*}
\bigg(\int_{\bbR^m}|\{x:\, |F(x,y)|>\lambda \}|^{p_2/p_1} dy\bigg)^{1/p_2}
&=  \frac{v_n^{1/p_1} \|w\|_{p_2}}{\lambda}.
\end{align*}
Therefore, $\|F\|_{L^{\vec p, \infty}} \le v_n^{1/p_1} \|w\|_{p_2}$.
On the other hand, it is easy to see that
 $\|F\|_{L^{p_2, \infty}(L^{p_1,\infty})} = \|1/ |x|^{n/p_1}\|_{L^{p_1,\infty}}\|w\|_{L^{p_2, \infty}}
\le v_n^{1/p_1} \|w\|_{p_2}$.

\end{proof}

It is easy to see that $f\otimes g  \in L^{p_2,\infty}(L^{p_1,\infty})\setminus\{0\}$
if and only if
$f\in L^{p_1,\infty}\setminus\{0\}$
    and
    $g\in L^{p_2,\infty}\setminus\{0\}$.
Next we consider the conditions for $f\otimes g\in L^{\vec p,\infty}$.

\begin{Theorem} \label{thm:weak norms}
Suppose that $0<p_1, p_2 < \infty$ and $m$ and $n$ are positive integers. We have

\begin{enumerate}

\item  If $f\in L^{p_1, \infty}(\bbR^n ) $ and $g\in L^{p_2}( \bbR^m) $, then $f\otimes g\in
L^{\vec p, \infty}(\bbR^n\times \bbR^m) $.

\item If $f\in L^{p_1} $,  $g\in L^{p_2, \infty}$ and $p_1\le p_2$, then $f\otimes g\in
L^{\vec p, \infty}$.


\item
    If $f\otimes g\in  L^{\vec p, \infty}$ and $f, g\ne 0$, then
    $f\in L^{p_1, \infty}$
    and
    $g\in L^{p_2, \infty}$. But $g$ need not to be in $g\in L^{p_2} $.
\end{enumerate}

\end{Theorem}

\begin{proof}
First, we show that
 $\|F\|_{L^{\vec p, \infty}} = \||F|^{p_1}\|_{L^{(1,p_2/p_1),\infty}}^{1/p_1}$ for any measurable function $F$ defined on $\bbR^n\times\bbR^m$
   and $\|g\|_{L^{p_2, \infty}} = \| |g|^{p_1}\|_{L^{p_2/p_1,\infty}}^{1/p_1}$
   for any measurable function $g$ defined on $\bbR^m$.

In fact, a simple computation shows that
\begin{align}
\|F\|_{L^{\vec p, \infty}}^{p_1}
&=  \sup_{\lambda>0} \lambda^{p_1} \left(\int_{\bbR^m} |\{x:\, |F(x,y)|> \lambda\}|^{p_2/p_1} dy\right)^{p_1/p_2}
    \nonumber \\
&=  \sup_{\lambda>0} \lambda^{p_1} \left(\int_{\bbR^m} |\{x:\, |F(x,y)|^{p_1}> \lambda^{p_1}\}|^{p_2/p_1} dy\right)^{p_1/p_2}
    \nonumber \\
&= \||F|^{p_1}\|_{L^{(1,p_2/p_1),\infty}}. \label{eq:F pq}
\end{align}

Similarly we can prove that $\|g\|_{L^{p_2, \infty}} = \| |g|^{p_1}\|_{L^{p_2/p_1,\infty}}^{1/p_1}$.

(i).
For any $\lambda>0$,
\[
  |\{x:\, f(x)g(y)>\lambda\}|^{1/p_1} \le \frac{\|f\|_{L^{p_1, \infty}} |g(y)|}{\lambda}.
\]
Hence
\[
  \left(\int_{\bbR^m}  |\{x:\, f(x)g(y)>\lambda\}|^{p_2/p_1} dy\right)^{1/p_2} \le
  \frac{\|f\|_{L^{p_1, \infty}} \|g\|_{L^{p_2}}}{\lambda}.
\]
Consequently,
\[
  \|f\otimes g\|_{L^{\vec p, \infty}} \le \|f\|_{L^{p_1, \infty}} \|g\|_{L^{p_2}}.
\]

(ii).\,\, By (\ref{eq:F pq}), it suffices to consider the case of $p_1=1\le p_2$.
For any $\lambda>0$,
since
\[
  \{(x,y)\!: |f(x)g(y)|>\lambda\} \subset \bigcup_{k\in\bbZ} \{(x,y)\!:  2^k \le |f(x)| < 2^{k+1}, |g(y)|>2^{-k-1}\lambda\},
\]
we have
\begin{align*}
&\hskip -2mm \lambda\|\chi^{}_{\{(x,y):\, |f(x)g(y)|>\lambda\}}\|_{L^{(1,p_2)}}  \\
&\le \lambda \sum_{k\in\bbZ}
     \left\|\chi^{}_{\{x: 2^k \le |f(x)| < 2^{k+1}\}}(x) \chi^{}_{\{y: |g(y)|>2^{-k-1}\lambda\}}(y)\right\|_{L^{(1,p_2)}} \\
&= \sum_{k\in\bbZ}
     2^k|\{x: 2^k \le |f(x)| < 2^{k+1}\}| \cdot   2^{-k} \lambda |\{y: |g(y)|>2^{-k-1}\lambda\}|^{1/p_2} \\
&\le 2\|f\|_1 \|g\|_{L^{p_2, \infty}}.
\end{align*}

(iii).\,\, Now we suppose that $f\otimes g \in L^{\vec p, \infty}(\bbR^n\times \bbR^m)$.
For any $\alpha, \beta>0$, we have
\[
   \{x:\, |f(x)|>\alpha\}\times \{y:\, |g(y)|>\beta\} \subset  \{(x,y):\, |f(x)g(y)|> \alpha\beta \}.
\]
Hence
\[
  \alpha
  |\{x\!: |f(x)|>\alpha\}|^{1/p_1} \!\cdot \beta |\{y\!: |g(y)|>\beta\}|^{1/p_2}
  \le \alpha\beta\left\|\chi^{}_{\{(x,y) : |f(x)g(y)|> \alpha\beta \}} \right\|_{L^{\vec p}}.
\]
Taking supremums on both sides, we get
\[
  \|f\|_{L^{p_1, \infty}} \|g\|_{L^{p_2, \infty}} \le \|f\otimes g\|_{L^{\vec p, \infty}}.
\]
Hence  $f\in L^{p_1, \infty}$ and $g\in L^{p_2, \infty}$.

Finally, we show that for
$f\otimes g \in L^{\vec p, \infty}(\bbR^n\times \bbR^m)$, $g$ need not to be in $L^{p_2}$.
Set $f(x)= 1/|x|^{2n /p_1}\chi^{}_{(1,\infty)}(|x|)$
and $g(y) = 1/|y|^{m/p_2}$ for $y\ne 0$.

Fix some $\lambda>0$. We have
\begin{align*}
|\{x:\, |f(x)g(y)| > \lambda\}|
&= \left|\Big\{x:\, |x|>1, |x|^{2n/p_1} < \frac{1}{\lambda |y|^{m/p_2}}\Big \}\right| \\
&= \begin{cases}
v_n \left(   \left(  \frac{1}{\lambda |y|^{m/p_2}}\right)^{p_1/2} - 1\right), & |y| < \frac{1}{\lambda^{p_2/m}}, \\
  0,  &\mathrm{others}.
\end{cases}
\end{align*}
Hence
\begin{align*}
\|f\otimes g\|_{L^{\vec p, \infty}}^{p_2}
&= \sup_{\lambda>0}\lambda^{p_2}  \int_{\bbR^m} |\{x:\, |f(x)g(y)| > \lambda\}|^{p_2/p_1} dy  \\
 &=\sup_{\lambda>0} \lambda^{p_2}  v_n^{p_2/p_1}  \int_{|y| < \frac{1}{\lambda^{p_2/m}}} \left(   \left(  \frac{1}{\lambda |y|^{m/p_2}}\right)^{p_1/2} - 1\right)^{p_2/p_1} dy  \\
 &\le\sup_{\lambda>0} 2 \lambda^{p_2}  v_n^{p_2/p_1}  \int_{|y| < \frac{1}{\lambda^{p_2/m}}} \left(    \frac{1}{\lambda^{p_2/2} |y|^{m/2}}  - 1\right) dy  \\
 &= 2 v_n^{p_2/p_1}v_m < \infty,
\end{align*}
where we use the fact that $(u-1)^{\alpha} \le 2 u^{\alpha} - 1$ for any $u>1$ and $\alpha>0$.
This completes the proof.
\end{proof}

Given $\vec p=(p_1, p_2)$, we compare the three mixed norms $L^{p_2,\infty}(L^{p_1})$, $L^{p_2}(L^{p_1,\infty})$,
and $L^{\vec p, \infty}$.

\begin{Theorem} \label{thm: mixed weak norms}
Suppose that $\vec p=(p_1, p_2)$. We have

\begin{enumerate}
\item For any measurable function $f$ defined on $\bbR^n\times \bbR^m$, we have
    \[
        \|f\|_{L^{\vec p,\infty }} \le \|f \|_{L^{p_2}(L^{p_1,\infty})}.
    \]

\item $L^{p_2,\infty}(L^{p_1}) \not\subset   L^{\vec p, \infty}$
and   $ L^{\vec p, \infty}\not\subset  L^{p_2,\infty}(L^{p_1}) $.
\end{enumerate}
\end{Theorem}

\begin{proof}
(i).\,\,
Let $E_{y,\lambda}$ be defined by (\ref{eq:E y}). Then we have
\begin{align*}
\|f\|_{L^{\vec p,\infty }}
&= \sup_{\lambda>0} \left\| |E_{y,\lambda}|^{1/p_1}\right\|_{L_y^{p_2}}
\le  \left\| \sup_{\lambda>0} |E_{y,\lambda}|^{1/p_1}\right\|_{L_y^{p_2}}
=\|f\|_{L^{p_2}(L^{p_1,\infty})}.
\end{align*}

(ii). Take some $f\in L^{p_1,\infty}\setminus L^{p_1}$ and $g\in L^{p_2}\setminus \{0\}$.
We see from Theorem~\ref{thm:weak norms}(i) that $f\otimes g \in L^{\vec p, \infty}\setminus L^{p_2,\infty}(L^{p_1})$.

On the other hand, set $F = \chi^{}_{E}$, where $E = \{(x,y):\, |x|^n \le |y|^{-mp_1/p_2}\}$. We have
\begin{align*}
 \|F\|_{L^{p_2,\infty}(L^{p_1})} \approx  \|  \frac{1}{|y|^{m/p_2}}\|_{L^{p_2,\infty}} < \infty
\end{align*}
and
\begin{align*}
 \|F\|_{L^{\vec p,\infty}}
= \|\chi^{}_E \|_{L^{\vec p}} \approx  \|  \frac{1}{|y|^{m/p_2}}\|_{L^{p_2}} = \infty.
\end{align*}
Hence $F\in L^{p_2,\infty}(L^{p_1}) \setminus L^{\vec p,\infty}$.
\end{proof}

\subsection{Convergence in weak norms}

In this subsection, we prove the completeness of both weak spaces.

First, we show that the Lebesgue monotone convergence theorem and Fatou's Lemma are true for weak norms. However, Lebesgue's dominated convergence theorem
fails for weak norms.

\begin{Theorem} \label{thm: monotone convergence theorem}
Let $W$ be either $L^{\vec p,\infty}$
or $L^{p_2,\infty}(L^{p_1,\infty})$, where $\vec p=(p_1, p_2)$ with $0<p_1,p_2\le \infty$.
Suppose that $\{f_k:\, k\ge 1\}$ is  a sequence of non-negative measurable functions such that
$f_k(x,y)$ $\le f_{k+1}(x,y)$, a.e., $k\ge 1$. Then we have
\begin{align*}
  \left\| \lim_{k\rightarrow\infty} f_k\right\|_W
    &= \lim_{k\rightarrow\infty} \left\|  f_k\right\|_W.
\end{align*}
\end{Theorem}

\begin{proof}
It suffices to prove the convergence in $L^{\vec p, \infty}$, since it implies the convergence in iterated weak spaces.

Denote $f(x,y)=\lim_{k\rightarrow\infty} f_k(x,y)$.
We consider only the case of $p_1, p_2<\infty$. Other cases can be proved similarly.

First, we assume that $\|f\|_{L^{\vec p,\infty}} <\infty$. For any $\varepsilon>0$, there exists some $\lambda_0>0$ such that
\[
     \lambda_0 \|\chi^{}_{\{|f|>\lambda_0\}} \|_{L^{\vec p}} \ge (1-\varepsilon) \|f\|_{L^{\vec p,\infty}}.
\]
Since $\{f_k:\,k\ge 1\}$ is increasing, we have
\[
 \lim_{k\rightarrow\infty} \lambda_0 \|\chi^{}_{\{|f_k|>\lambda_0\}} \|_{L^{\vec p}}
 = \lambda_0 \|\chi^{}_{\{|f|>\lambda_0\}} \|_{L^{\vec p}}\ge (1-\varepsilon) \|f\|_{L^{\vec p,\infty}}.
\]
Hence for $k$ sufficiently large, we have
\[
  \lambda_0 \|\chi^{}_{\{|f_k|>\lambda_0\}} \|_{L^{\vec p}} \ge (1-\varepsilon)^2 \|f\|_{L^{\vec p,\infty}}.
\]
Therefore,
\[
  \lim_{k\rightarrow\infty} \left\|  f_k\right\|_{L^{\vec p, \infty}} \ge (1-\varepsilon)^2 \|f\|_{L^{\vec p,\infty}}.
\]
By letting $\varepsilon\rightarrow 0$, we get
\[
  \lim_{k\rightarrow\infty} \left\|  f_k\right\|_{L^{\vec p, \infty}} \ge  \|f\|_{L^{\vec p,\infty}}.
\]
On the other hand, since $|f_k|\le |f|$ almost everywhere,
the reverse inequality is obvious. Hence  $\lim_{k\rightarrow\infty} \left\|  f_k\right\|_{L^{\vec p, \infty}} =  \|f\|_{L^{\vec p,\infty}}$.
The case $\|f\|_{L^{\vec p,\infty}} =\infty$ can be proved similarly.
\end{proof}

With the above theorem, we can prove Fatou's lemma for weak norms, which improves \cite[Exercise 1.1.12 (d)]{Grafakos2008} slightly (we remove an extra constant).

\begin{Theorem}\label{thm: Fatou}
Let $W$ be either $L^{\vec p,\infty}$
or $L^{p_2,\infty}(L^{p_1,\infty})$, where $\vec p=(p_1, p_2)$ with $0<p_1,p_2\le \infty$.
Let $\{f_k:\, k\ge 1\}$ be a sequence of non-negative measurable functions.
Then we have
\begin{align*}
  \left\| \liminf_{k\rightarrow\infty} f_k\right\|_{W}
    &\le  \liminf_{k\rightarrow\infty} \left\|  f_k\right\|_{W}.
\end{align*}
\end{Theorem}

\begin{proof}
Let $g_k = \inf_{l\ge k} f_l$. Then
  $\{g_k:\, k\ge 1\}$ is  increasingly convergent to $\liminf_{k\rightarrow\infty} f_k$. It follows from the monotone convergence theorem that
\[
 \left\| \liminf_{k\rightarrow\infty} f_k\right\|_{W}
 = \left\| \lim_{k\rightarrow\infty} g_k\right\|_{W}
 = \liminf_{k\rightarrow\infty}\left\|  g_k\right\|_{W}
 \le \liminf_{k\rightarrow\infty}\left\|  f_k\right\|_{W}.
\]
\end{proof}

However, the dominated convergence theorem fails in weak norm spaces.
For example, set $f_0(x) = 1/|x|^{n/p_1}$ and $f_k(x) = f_0(x) \chi_{[k,\infty]}(|x|)$.
Take some $g\in L^{p_2}\setminus\{0\}$.  We have
$\lim_{k\rightarrow\infty} f_k\otimes g (x,y) =0$.
Moreover, we see from Theorem~\ref{thm:weak norms}
that $f_k\otimes g \le f_0\otimes g \in L^{\vec p,\infty} \cap L^{p_2,\infty}(L^{p_1,\infty})$.
But
\[
  \|f_k\otimes g\|_{L^{\vec p,\infty}} =  \|f_k\otimes g\|_{L^{p_2,\infty}(L^{p_1,\infty})} =v_n^{1/p_1}\|g\|_{L^{p_2}},
  \,\, k\ge 1.
\]

\subsubsection{Convergence in measure and almost everywhere} \label{subsubsection:convergence}

It is known that if $\{f_k:\, k\ge 1\}$ is convergent in $L^p$ or $L^{p,\infty}$, then it is convergent in measure.
However, it is not true for mixed norm. Specifically, neither the strong convergence nor
 the weak convergence in mixed norm spaces implies the convergence in measure.

For example, set
$f_k = \chi^{}_{E_k}$, where
\[
  E_k = \begin{cases}
   \{(x,y):\, |x|^n < 1/|y|^m, 0<|y|<1/k\}, & p_1>p_2, \\
   \{(x,y):\, |x|^n < 1/|y|^m, |y|>k\}, & p_1<p_2.
  \end{cases}
\]
Then we have  $\|f_k\|_{L^{\vec p}} \rightarrow 0$.
Therefore,
\[
  \lim_{k\rightarrow\infty}\|f_k\|_{L^{\vec p,\infty}}
=\lim_{k\rightarrow\infty}\|f_k\|_{L^{p_2,\infty}(L^{p_1,\infty})}=0.
\]
However, for any $0<\varepsilon < 1$, $|\{(x,y):\, |f_k(x,y)|>\varepsilon\}| = \infty$.

Nevertheless, it was shown in \cite[Theorem 1]{Benedek1961} that if $\{f_k:\,k\ge 1\}$ is convergent to $f$
in $L^{\vec p}$, then it contains a subsequence  convergent    almost everywhere to $f$.
We show that the same is true for weak norms.

\begin{Theorem}\label{thm:pointwise convergence}
Let $\{f_k:\, k\ge 1\}$ be a sequence of measurable functions which is convergent to some $f$ in $L^{\vec p,\infty}$
or $L^{p_2,\infty}(L^{p_1,\infty})$,  where $\vec p=(p_1, p_2)$ with $0<p_1,p_2\le \infty$.
Then it contains a subsequence  convergent almost   everywhere to $f$.
\end{Theorem}

\begin{proof}
(i).\,\, First we assume that   $\lim_{k\rightarrow\infty}
\|f_k - f\|_{L^{\vec p, \infty}} = 0$. For any $\lambda>0$, we have
\[
  \lim_{k\rightarrow\infty }  \left\|\chi^{}_{\{|f_k-f|>\lambda\}}\right\|_{L^{\vec p}} =0,
\]
where $\{|f|>\lambda\}$ stands for $\{(x,y):\, |f(x,y)|>\lambda\}$.
By \cite[Theorem 1]{Benedek1961}, there is
some subsequence of $\{f_k:\, k\ge 1\}$, denoted by $\{f_{1,k}:\, k\ge 1\}$, such that
\[
   \lim_{k\rightarrow \infty} \chi^{}_{\{|f_{1,k}-f|>\lambda\}}(x,y) =0,\qquad
   (x,y)\in \bbR^n\times \bbR^m\setminus E_{\lambda},
\]
where
$|E_{\lambda}|=0$.
Consequently,
\[
   \limsup_{k\rightarrow \infty}  |f_{1,k}(x,y)-f(x,y)|\le \lambda,\qquad
   (x,y)\in \bbR^n\times \bbR^m \setminus E_{\lambda}.
\]

By setting $\lambda = 1$, $1/2$, $\ldots$, $1/2^l$, $\ldots$, we get a sequence of zero measured sets
$\{E_l:\, l\ge 1\}$ and subsequences $\{f_{l,k}:\,k\ge 1\}$ of $\{f_k:\, k\ge 1\}$,
such that $\{f_{l,k}:\,k\ge 1\}$ is a subsequence of
$\{f_{l-1,k}:\,k\ge 1\}$ for $l\ge 2$ and for each $l\ge 1$,
\[
     \limsup_{k\rightarrow \infty}  |f_{l,k}(x,y)-f(x,y)|\le \frac{1}{2^l},\qquad
   (x,y)\in \bbR^n\times \bbR^m\setminus E_{l}.
\]
Set $E=\cup_{l\ge 1} E_l$. Then for $(x,y)\not\in E$, we have
\[
     \limsup_{k\rightarrow \infty}  |f_{k,k}(x,y)-f(x,y)|\le \frac{1}{2^l},\qquad\forall l\ge 1.
\]
Hence
\[
     \lim_{k\rightarrow \infty}  |f_{k,k}(x,y)-f(x,y)|=0,\qquad (x,y)\in\bbR^n\times \bbR^m\setminus E.
\]
This proves the conclusion for the mixed weak norm.

(ii).\,\,  Next we assume that $\lim_{k\rightarrow\infty} \|f_k-f\|_{L^{p_2,\infty}(L^{p_1,\infty})}=0$.
We consider only the case of $0<p_1,p_2<\infty$. Other cases can be proved similarly.

For any $\alpha, \beta>0$, we have
\begin{align*}
&
  \hskip -3em \left|\left\{y:\, \big|\{x:\, |f_k(x,y)-f(x,y)|>\alpha\}\big|^{p_2/p_1} >\beta\right\}\right|\\
&\le   \left|\left\{y:\, \|f_k(x,y)-f(x,y)\|_{L^{p_1,\infty}} >\alpha\beta^{1/p_2}\right\}\right|\\
&\le \frac{1}{\alpha^{p_2}\beta} \|f_k-f\|_{L^{p_2,\infty}(L^{p_1,\infty})}^{p_2}\\
&\rightarrow 0,\qquad \mathrm{as}\,\, k\rightarrow\infty.
\end{align*}
Hence there is a subsequence $\{f_{k_l}:\, l\ge 1\}$ such that
\[
  \left|\left\{y:\, \Big|\{x:\, |f_{k_l}(x,y)-f(x,y)|>\frac{1}{2^l}\}\Big|^{p_2/p_1}
  > \frac{1}{2^{lp_2/p_1}}\right\}\right| < \frac{1}{2^l},\quad l\ge 1.
\]
Let $E=\cap_{i\ge 1}E_i$ and $F=\cap_{i\ge 1}F_i$, where
\begin{align*}
   E_i &= \bigcup_{l=i}^{\infty} \left\{y:\, \Big|\{x:\, |f_{k_l}(x,y)- f (x,y)|>\frac{1}{2^l}\}\Big|^{p_2/p_1}
  > \frac{1}{2^{lp_2/p_1}}\right\}, \\
   F_i &= \bigcup_{l=i}^{\infty} \left\{(x,y):\, |f_{k_l}(x,y)- f (x,y)|>\frac{1}{2^l}  \right\}.
\end{align*}
Then we have $|E_i|\le 1/2^{i-1}$ and therefore $|E|=0$.

Take some $y\not\in E$. Then there is some $i\ge 1$ such that $y\not\in E_i$. Consequently,
for $l\ge i$,
\[
  \Big|\{x:\, |f_{k_l}(x,y)-f(x,y)|>\frac{1}{2^l}\}\Big| \le \frac{1}{2^l}.
\]
Hence $|\{x:\, (x,y)\in F_j\}| \le 1/2^{j-1}$ for $j\ge i$. Therefore,
\[
  \Big| \big\{x:\, (x,y) \in F  \big\}\Big | = 0.
\]
It follows that
\[
 |F|= \iint_{\bbR^n\times \bbR^m} \chi^{}_F (x,y) dxdy
  = \int_{\bbR^m\setminus E} dy \int_{\bbR^n} \chi^{}_F  (x,y) dx = 0.
\]
For $(x,y)\not\in (\bbR^n\times E) \cup F$, there is some $i\ge 1$ such that
$(x,y)\not\in F_i$. Hence
\[
  |f_{k_l}(x,y) - f(x,y)| \le \frac{1}{2^l}, \qquad l\ge i.
\]
That is, $\{f_{k_l}:\, l\ge 1\}$ converges to $f$ almost everywhere.
\end{proof}

\subsubsection{Completeness of weak norm spaces}

It is known that both $L^p$ and weak $L^p$ are complete \cite{Grafakos2008}.
We show that the same is true for weak norms  in mixed  spaces.

\begin{Theorem}\label{thm:Completeness}
Let $W$ be either $L^{\vec p,\infty}$
or $L^{p_2,\infty}(L^{p_1,\infty})$, where $\vec p=(p_1, p_2)$ with $0<p_1,p_2\le \infty$.
Let $\{f_k:\,k\ge 1\}$ be a Cauchy sequence in $W$, that is,
\[
    \lim_{k,l\rightarrow\infty} \|f_k -f_l\|_W = 0.
\]
Then there is some $f\in W$ such that
$\lim_{k\rightarrow\infty}\|f-f_k\|_W= 0$.
\end{Theorem}

Before giving a proof of the above theorem, we introduce a preliminary result.

\begin{Lemma} \label{Lm:weak norm sum}
Let $a_i$, $\tilde a_i$, $b_i$ and  $c_i$ be positive numbers, $1\le i\le k\le \infty$,  such that
$\sum_{i=1}^k a_i = \sum_{i=1}^k \tilde a_i = 1$, $b_1=\ldots=b_k=1$ for $1\le p_1\le \infty$ and $\sum_{i=1}^k b_i^{p_1/(1-p_1)}\le 1$ for $0<p_1<1$,
and  $c_1=\ldots=c_k=1$ for $1\le p_2\le \infty  $ and $\sum_{i=1}^k c_i^{p_2/(1-p_2)}\le 1$ for $0<p_2<1$.
Then    we have
\begin{align*}
  \|f_1+\ldots + f_k\|_{L^{\vec p,\infty}} &\le
    \sum_{i=1}^k
    \frac{1}{a_ib_ic_i} \|f_i\|_{L^{\vec p,\infty}},  \\
  \|f_1+\ldots + f_k\|_{L^{p_2,\infty}(L^{p_1,\infty})} &\le
    \sum_{i=1}^k \frac{1}{a_i\tilde a_i b_ic_i} \|f_i\|_{L^{p_2,\infty}(L^{p_1,\infty})}.
\end{align*}
\end{Lemma}

\begin{proof}
We prove only the first inequality. The second one can be proved similarly.

For any $\lambda>0$, we have
\[
\{|f_1+\ldots+f_k|>\lambda\} \subset \{|f_1|>a_1\lambda\}\cup\ldots \cup\{|f_k|>a_k\lambda\}.
\]
Hence
\[
\chi^{}_{\{|f_1+\ldots+f_k|>\lambda\}}
\le  \sum_{i=1}^k \chi^{}_{\{|f_i|>a_i\lambda\}}.
\]
If $p_1\ge 1$, then
\[
   \|\chi^{}_{\{|f_1+\ldots+f_k|>\lambda\}}\|_{L_x^{p_1}}
  \le
    \sum_{i=1}^k  \|\chi^{}_{\{|f_i |>a_i\lambda\}}\|_{L_x^{p_1}}.
\]
If $0<p_1<1$, then
\[
   \|\chi^{}_{\{|f_1+\ldots+f_k|>\lambda\}}\|_{L_x^{p_1}}^{p_1}
  \le
    \sum_{i=1}^k  \|\chi^{}_{\{|f_i |>a_i\lambda\}}\|_{L_x^{p_1}}^{p_1}.
\]
Set $q_1 = 1/p_1$. We see from H\"older's inequality that
\begin{align}
\|\chi^{}_{\{|f_1+\ldots+f_k|>\lambda\}}\|_{L_x^{p_1}}
&\le\left( \sum_{i=1}^k  \|\chi^{}_{\{|f_i |>a_i\lambda\}}\|_{L_x^{p_1}}^{p_1}\right)^{q_1} \nonumber \\
&\le\left( \sum_{i=1}^k \frac{1}{b_i} \|\chi^{}_{\{|f_i |>a_i\lambda\}}\|_{L_x^{p_1}} \right)
   \left( \sum_{i=1}^k b_i^{q'_1/q_1}\right)^{q_1/q'_1} \nonumber \\
&\le \sum_{i=1}^k \frac{1}{b_i} \|\chi^{}_{\{|f_i |>a_i\lambda\}}\|_{L_x^{p_1}}. \label{eq:e5}
\end{align}
With similar arguments we get
\[
\left\| \|\chi^{}_{\{|f_1+\ldots+f_k|>\lambda\}}\|_{L_x^{p_1}} \right\|_{L_y^{p_2}}
  \le
    \sum_{i=1}^k \frac{1}{b_ic_i} \left\|   \|\chi^{}_{\{|f_i |>a_i\lambda\}}\|_{L_x^{p_1}} \right\|_{L_y^{p_2}}.
\]
Hence
\[
\lambda  \|\chi^{}_{\{|f_1+\ldots+f_k|>\lambda\}}\|_{L^{\vec p}}
  \le
    \sum_{i=1}^k \frac{1}{b_ic_i}   \lambda \|\chi^{}_{\{|f_i |>a_i\lambda\}}\|_{L^{\vec p}}.
\]
Taking supremums on both sides, we get the conclusion as desired.
\end{proof}

\begin{proof}[Proof of Theorem~\ref{thm:Completeness}]
We prove the conclusion only for $W=L^{\vec p,\infty}$. The other case can be proved similarly.

For $k\ge 1$, let
\begin{align*}
&b_k = \begin{cases}
1, & p_1\ge 1, \\
\frac{b_0}{2^k}, & 0<p_1<1,
\end{cases}
&c_k = \begin{cases}
1, & p_2\ge 1, \\
\frac{c_0}{2^k}, & 0<p_2<1,
\end{cases}
\end{align*}
where $b_0$ and $c_0$ are constants such that $\sum_{k=1}^{\infty} b_k^{p_1/(1-p_1)} = \sum_{k=1}^{\infty} c_k^{p_2/(1-p_2)} =1$ for $0<p_1,p_2<1$.

We see from the hypothesis that there is a subsequence of
$\{f_k:\, k\ge 1\}$, say $\{f_{k_l}:\, l\ge 1\}$, such that
\[
\left\|  f_{k_{l+1}} - f_{k_l}\right\|_{L^{\vec p,\infty}} \le \frac{b_l c_l}{2^{2l}}.
\]
Let
\[
  g_i = \sum_{l=i}^{\infty} |f_{k_{l+1} - f_{k_l}}|.
\]
We see from Lemma~\ref{Lm:weak norm sum}  that
\[
  \|g_i\|_{L^{\vec p,\infty}}
    \le  \sum_{l=i}^{\infty} \frac{2^l}{b_lc_l}
     \left\|  f_{k_{l+1} - f_{k_l}} \right\|_{L^{\vec p,\infty}}
    \le \frac{1}{2^{i-1}}\rightarrow 0.
\]
By Theorem \ref{thm:pointwise convergence}, $\{g_i:\,i\ge 1\}$ contains a subsequence convergent almost everywhere.
Since $\{g_i:\,i\ge 1\}$ is decreasing, it is convergent to zero almost everywhere.

Suppose that for some $(x,y)$, $g_i(x,y)$ tends to zero as $i$ tends to the infinity.
Since
\[
   |f_{k_m}(x,y) - f_{k_l}(x,y)| \le g_i(x,y),\qquad m>l\ge i,
\]
 $\{f_{k_l}(x,y):\,l\ge 1\}$ is a Cauchy sequence in $\bbR$ for almost all $(x,y)\in\bbR^{2n}$. Denote
$f(x,y)=\lim_{l\rightarrow \infty} f_{k_l}(x,y)$. Then $f$ is well defined almost everywhere.

For any $i\ge 1$, we have
\begin{align*}
\|f_{k_i} - f\|_{L^{\vec p,\infty}}
 &\le  \|g_i\|_{L^{\vec p,\infty}}
 \le \frac{1}{2^{i-1}}.
\end{align*}
Hence $\lim_{i\rightarrow\infty} \|f_{k_i} - f\|_{L^{\vec p,\infty}} = 0$.
Since $\{f_k:\,k\ge 1\}$ is a Cauchy sequence in $L^{\vec p,\infty}$, it is easy to see that
$\lim_{k\rightarrow\infty} \|f_{k} - f\|_{L^{\vec p,\infty}} = 0$. This completes the proof.
\end{proof}

It is known that if a sequence of measurable functions $\{f_n:\, n\ge 1\}$ is convergent in $L^{\vec p}$, then we have
$\lim_{n\rightarrow \infty}\|f_n\|_{L^{\vec p}} = \|f\|_{L^{\vec p}}$. We show that the same is true for weak norms.

\begin{Theorem}\label{thm:convergence norm}
Let $W$ be either $L^{\vec p,\infty}$
or $L^{p_2,\infty}(L^{p_1,\infty})$, where  $\vec p = (p_1, p_2)$ with $0<p_1,p_2\le \infty$.
Suppose that $\lim_{k\rightarrow\infty} \|f_k-f\|_{W} = 0$.
Then we have $\lim_{k\rightarrow\infty} \|f_k \|_{W} = \|f \|_{W}$.
\end{Theorem}

\begin{proof}
We prove the conclusion only for $W=L^{\vec p,\infty}$. The other case can be proved similarly.

For $1\le p_1\le \infty$, set $b_1=b_2=1$. And for $0<p_1<1$, set $0<b_1<1$ and
$b_2 = (1-b_1^{p_1/(1-p_1)})^{(1-p_1)/p_1}$.   Define $c_1$ and $c_2$ similarly (replacing $(p_1,b_1,b_2)$ by $(p_2,c_1,c_2)$).

For any $0<a<1$, we see from Lemma~\ref{Lm:weak norm sum} that
\[
  \|f_k \|_{L^{\vec p, \infty}} = \|f_k -f + f\|_{L^{\vec p, \infty}}
  \le \frac{1}{(1-a)b_2c_2} \|f_k -f\|_{L^{\vec p, \infty}} + \frac{1}{ab_1c_1} \|f\|_{L^{\vec p, \infty}}.
\]
Hence
\[
  \limsup_{k\rightarrow\infty}
   \|f_k \|_{L^{\vec p, \infty}}
  \le  \frac{1}{ab_1c_1} \|f\|_{L^{\vec p, \infty}}.
\]
Letting $a, b_1,c_1\rightarrow 1$, we get
\[
  \limsup_{k\rightarrow\infty}
   \|f_k \|_{L^{\vec p, \infty}}
  \le   \|f\|_{L^{\vec p, \infty}}.
\]

On the other hand, since
\[
  \|f \|_{L^{\vec p, \infty}} = \|f -f_k + f_k\|_{L^{\vec p, \infty}}
  \le \frac{1}{(1-a)b_2c_2} \|f -f_k\|_{L^{\vec p, \infty}} + \frac{1}{ab_1c_1} \|f_k\|_{L^{\vec p, \infty}},
\]
by letting $k\rightarrow\infty$ and $a,b_1,c_1\rightarrow 1$ successively, we get
\[
   \|f\|_{L^{\vec p, \infty}} \le \liminf_{k\rightarrow\infty}
   \|f_k \|_{L^{\vec p, \infty}}.
\]
Now the conclusion follows.
\end{proof}

In \cite{Benedek1961}, the Riesz theorem
for mixed norm Lebesgue spaces was proved.
It says that if $\lim_{k\rightarrow\infty }\|f_k\|_{L^{\vec p}} = \|f\|_{L^{\vec p}}$
and $\lim_{k\rightarrow\infty} f_k(x,y) = f(x,y)$ almost everywhere on $\bbR^{2n}$,
where $\vec p = (p_1, p_2)$ with $1\le p_1, p_2<\infty$,
then we have
\[
  \lim_{k\rightarrow\infty }\|f_k - f\|_{L^{\vec p}} = 0.
\]

Whenever weak norms are considered, the above conclusion fails.
For example, set $f_0(x) = 1/|x|^{n/p_1}$ and $f_k(x) = f_0(x) \chi_{[0,k]}(|x|)$.
Take some $g\in L^{p_2}\setminus\{0\}$.  We have
\[
    \lim_{k\rightarrow\infty} f_k\otimes g (x,y) = f_0(x)g(y)
\]
and
\[
    \lim_{k\rightarrow\infty} \|f_k\otimes g\|_{L^{\vec p,\infty}}=
    \lim_{k\rightarrow\infty} \|f_k\otimes g\|_{L^{p_2,\infty}(L^{p_1,\infty})}=v_n^{1/p_1}\|g\|_{L^{p_2}}.
\]
However, for any $k\ge 1$,
\[
    \|f_k\otimes g - f_0\otimes g\|_{L^{\vec p,\infty}}=
     \|f_k\otimes g - f_0\otimes g\|_{L^{p_2,\infty}(L^{p_1,\infty})}
     =v_n^{1/p_1}\|g\|_{L^{p_2}}.
\]
Hence $\{f_k\otimes g:\, k\ge 1 \}$ is not convergent to $f_0\otimes g$ in $L^{\vec p,\infty}$ or $L^{p_2,\infty}(L^{p_1,\infty})$.

\subsection{Convergence in truncated norm}

We see from the previous subsection that while a  sequence convergent in mixed norm contains a subsequence convergent almost everywhere,
it might not contain a subsequence convergent in measure.
This prompts us to consider the following substitution of convergence in measure.

\begin{Definition}
Suppose that $\vec p=(p_1, p_2)$ with $0<p_1, p_2\le \infty$.
We say that a sequence of measurable functions $\{f_k:\, k\ge 1\}$ is convergent to some $f$ in truncated   $L^{\vec p}$   norm if for any $\lambda>0$,
\[
  \lim_{k\rightarrow\infty} \| \chi^{}_{\{|f_k-f|>\lambda\}} \|_{L^{\vec p}} = 0.
\]
\end{Definition}

We see from the definition that convergence in truncated $L^{\vec p}$ norm is the same as convergence in measure if $p_1 = p_2$.
However, they are not equivalent if $p_1\ne p_2$. For an example, see Subsubsection~\ref{subsubsection:convergence}.

It is easy to see that if  $\{f_k:\, k\ge 1\}$ is convergent to some $f$ in $L^{\vec p,\infty}$,
then it is also convergent to $f$ in truncated     norm.
Moreover, we see from the proof of  Theorem~\ref{thm:pointwise convergence} that
a sequence convergent in truncated    norm always contains a subsequence convergent almost everywhere.

The following is a characterization of convergence in truncated norm, which shows also the completeness of such convergence.

\begin{Theorem}
Suppose that $\vec p=(p_1, p_2)$ with $0<p_1, p_2\le \infty$.
Let  $\{f_k:\, k\ge 1\}$  be a sequence of measurable functions.
Then   $\{f_k:\, k\ge 1\}$ is convergent to some  $f$ in truncated    norm if and only if
$\{f_k:\, k\ge 1\}$
is a Cauchy sequence in truncated norm, i.e.,
\[
  \lim_{k,l\rightarrow\infty} \| \chi^{}_{\{|f_k-f_l|>\lambda\}} \|_{L^{\vec p}} = 0,\qquad \forall \lambda>0.
\]
\end{Theorem}

\begin{proof}
First we prove the Necessity.
Suppose that   $\{f_k:\, k\ge 1\}$ is convergent to $f$ in truncated norm.
For any $\lambda>0$, we have
\[
  \{|f_k - f_l| >\lambda\} \subset \{|f_k - f| >\frac{\lambda}{2}\} \cup \{|f_l - f| >\frac{\lambda}{2}\}.
\]
Hence
\[
  \chi^{}_{\{|f_k - f_l| >\lambda\}} \le \chi^{}_{\{|f_k - f| >\lambda/2\}} + \chi^{}_{\{|f_l - f| >\lambda/2\}}.
\]
Therefore,
\begin{align*}
   \lim_{k,l\rightarrow\infty} \| \chi^{}_{\{|f_k-f_l|>\lambda\}} \|_{L^{\vec p}}
   &\le  \lim_{k\rightarrow\infty} \frac{1}{b_1 c_1}\| \chi^{}_{\{|f_k-f|>\lambda/2\}} \|_{L^{\vec p}} \\
   &\qquad
      + \lim_{l\rightarrow\infty} \frac{1}{b_2 c_2}\| \chi^{}_{\{|f_l-f|>\lambda/2\}} \|_{L^{\vec p}}\\
&
  =0,
\end{align*}
where $b_1,b_2,c_1$ and $c_2$ are defined as in the proof of Theorem~\ref{thm:convergence norm}.

Next we prove the sufficiency, for which we give only a sketch since it is similar to the proof of Theorem~\ref{thm:Completeness}.

Define $\{b_k:\, k\ge 1\}$ and $\{c_k:\, k\ge 1\}$ as in the proof of Theorem~\ref{thm:Completeness}.
Then there is some subsequence $\{f_{k_l}:\, l\ge 1\}$ such that
\[
\left\| \chi^{}_{\{ |f_{k_{l+1} - f_{k_l}}|>1/2^l \}}\right\|_{L^{\vec p}} \le \frac{b_lc_l}{2^{l}}.
\]
Let
\[
  g_i = \sum_{l=i}^{\infty} \chi^{}_{\{ |f_{k_{l+1} - f_{k_l}}|>1/2^l \}}.
\]
Similar arguments as in the proof of Lemma~\ref{Lm:weak norm sum} show that
\[
  \|g_i\|_{L^{\vec p}}
    \le  \sum_{l=i}^{\infty} \frac{1}{b_lc_l}
     \left\| \chi^{}_{\{ |f_{k_{l+1} - f_{k_l}}|>1/2^l \}}\right\|_{L^{\vec p}}
    \le \frac{1}{2^{i-1}}\rightarrow 0.
\]
Hence $\{g_i:\,i\ge 1\}$ contains a subsequence convergent almost everywhere.
Since $\{g_i:\,i\ge 1\}$ is decreasing, it is convergent to zero almost everywhere.

Suppose that for some $(x,y)$, $g_i(x,y)$ tends to zero as $i$ tends to the infinity.
Then for $i$ sufficiently large, $g_i(x,y)<1/2$. In other words,
\[
   |f_{k_{l+1}}(x,y) - f_{k_l}(x,y)| \le \frac{1}{2^l},\qquad l\ge i.
\]
Hence $\{f_{k_l}(x,y):\,l\ge 1\}$ is a Cauchy sequence in $\bbR$. Denote
$f(x,y)=\lim_{l\rightarrow \infty} f_{k_l}(x,y)$. Then $f$ is well defined almost everywhere.

Fix some $\lambda>0$. Whenever $1/2^{i-1} < \lambda$,
we have
\[
  \{|f_{k_i}-f| >\lambda\} \subset \bigcup_{l=i}^{\infty} \{ | f_{k_{l+1}} - f_{k_l}|  > \frac{1}{2^l}\}.
\]
Similarly to (\ref{eq:e5}) we get
\begin{align*}
\left\|\chi^{}_{\{|f_{k_i}-f| >\lambda\}}\right\|_{L^{\vec p}}
 & \le \sum_{l=i}^{\infty} \frac{1}{b_lc_l}  \left\|\chi^{}_{\{ | f_{k_{l+1}} - f_{k_l}|  > 1/2^l\}} \right\|_{L^{\vec p}}
 \le \frac{1}{2^{i-1}}.
\end{align*}
Hence $\{f_{k_l}:\, l\ge 1\}$ is convergent to $f$ in truncated norm. Since $\{f_k:\,k\ge 1\}$ is a Cauchy sequence, it is easy to see that
$\{f_k:\,k\ge 1\}$ is also convergent to $f$ in truncated norm.
\end{proof}

\subsection{H\"older's inequality}

It is well known that H\"older's inequality holds for both $L^p$ and $L^{p,\infty}$.
For mixed norms, it was shown in \cite{Benedek1961} that if $1\le p_i\le \infty$, $i=1,2$,
then we have
\[
  \|fg\|_{L^{\vec 1}} \le \|f\|_{L^{\vec p}} \|g\|_{\vec p'},
\]
where $\vec p'=(p'_1, p'_2)$.

Suppose that
\[
  \frac{1}{r_i} = \frac{1}{p_i} + \frac{1}{q_i},\qquad i=1,2.
\]
The previous inequality can be rewritten as
\begin{equation}\label{eq: Holder}
\|fg\|_{L^{\vec r}} \le \|f\|_{L^{\vec p}} \|g\|_{\vec q},
\end{equation}
where $f$ and $g$ are arbitrary measurable functions.

Using  H\"older's inequality for weak spaces~\cite[Exercise 1.1.15]{Grafakos2008},
we get H\"older's inequality for iterated weak norms,
for which we omit the proof.
\begin{Theorem}
Suppose that $0<p_i, q_i, r_i\le \infty$ and that $1/r_i = 1/p_i + 1/q_i$, $i=1,2$. Then we have
\[
  \| fg\|_{L^{r_2,\infty}(L^{r_1,\infty})}
  \le  C_{\vec p, \vec q}
  \| f\|_{L^{p_2,\infty}(L^{p_1,\infty})} \| g\|_{L^{q_2,\infty}(L^{q_1,\infty})},
\]
where $C_{\vec p, \vec q} =
 \prod_{i=1}^{2} \left(p_i/r_i\right)^{1/p_i} \left(q_i/r_i\right)^{1/q_i} $ with suitable changes for
 $p_i=\infty$ or $q_i=\infty$.
\end{Theorem}

However,   for mixed weak norms, H\"older's  inequality  is  true  only for very special cases.
The following is a complete characterization of indices for which H\"older's inequality
is true on mixed weak spaces.

\begin{Theorem}\label{thm: holder mixed weak norm}
Suppose that $1/r_i  = 1/p_i + 1/q_i$, $i=1,2$,
where $0<p_1$, $p_2$, $q_1$, $q_2\le \infty$.
Then there exists some constant $C_{\vec p, \vec q}<\infty$ such that
\[
  \|fg\|_{L^{\vec r,\infty}} \le  C_{\vec p, \vec q}
    \|f\|_{L^{\vec p,\infty}} \|g\|_{L^{\vec q,\infty}}, \qquad \forall f,g,
\]
if and only if  $p_1q_2=p_2q_1$.

When the condition is true, we have
\[
  C_{\vec p, \vec q} = \begin{cases}
  \max\{1,2^{1/r_1-1/r_2}\}
  \frac{p_2^{1/p_2}q_2^{1/q_2}}{r_2^{1/r_2}},
           & 0< p_1,p_2,q_1,q_2<\infty,\\
  \max\{1,2^{1/r_1-1}\}
  \frac{p_1^{r_1/p_1} q_1^{r_1/q_1}}{r_1},
      & p_2=q_2=\infty, 0<p_1,p_2<\infty, \\
  \frac{p_2^{1/p_2}q_2^{1/q_2}}{r_2^{1/r_2}},
       & p_1=q_1=\infty, 0<p_2,q_2<\infty,\\
  1, & \vec p =(\infty,\infty) \mathrm{\,\, or\,\,}
     \vec q=(\infty, \infty).
  \end{cases}
\]
\end{Theorem}

To prove   H\"older's inequality, we need the following preliminary  result,
for which we omit the proof.

\begin{Lemma}\label{Lm: inequalities}

\begin{enumerate}
\item For any $\alpha>0$ and $a,b>0$, we have
\[
  (a+b)^{\alpha} \le \max\{2^{\alpha-1}, 1\} (a^{\alpha} + b^{\alpha}).
\]

\item  For $0\le \theta \le1$ and $a,b \ge0$, we have
\[
  a + b \ge  \frac{a^{\theta} b^{1-\theta}}{\theta^{\theta} (1-\theta)^{1-\theta}},
\]
and the equality is true if and only if
$a/\theta = b/(1-\theta)$.

\end{enumerate}
\end{Lemma}

\begin{proof}[Proof of Theorem~\ref{thm: holder mixed weak norm}]
We prove only the sufficiency. The necessity can be found in Example~\ref{Ex: holder}.

First, we consider the case of $0<p_1,p_2,q_1,q_2<\infty$
with $p_1/q_1 = p_2/q_2$.
For any $\lambda, a>0$, we have
\begin{equation}\label{eq:H:fg}
  \{(x,y):f(x,y)g(x,y)>\lambda\}
  \subset \{(x,y): f(x,y)>\frac{\lambda}{a}\} \bigcup
  \{(x,y): g(x,y)>a\}.
\end{equation}
Denote $E_{y,\lambda,f} = \{x:\, f(x,y)>\lambda\}$. We see from the above formula that
$E_{y,\lambda,fg} \subset E_{y,\lambda/a,f} \cup E_{y,a,g}$.
Hence
\begin{equation}\label{eq:H:e2}
|E_{y,\lambda,fg}| \le |E_{y,\lambda/a,f}|+ | E_{y,a,g}|.
\end{equation}
Note that $r_2/r_1 = p_2/p_1 = q_2/q_1$. We see from Lemma~\ref{Lm: inequalities} that there is
some constant $C_{p_2/p_1}$ such that
\[
|E_{y,\lambda,fg}|^{r_2/r_1} \le C_{p_2/p_1} \left(
 |E_{y,\lambda/a,f}|^{p_2/p_1}+ | E_{y,a,g}|^{q_2/q_1}\right).
\]
Hence
\begin{align*}
 \int_{\bbR^m} |E_{y,\lambda,fg}|^{r_2/r_1} dy
&\le    C_{p_2/p_1}  \left(
\int_{\bbR^m}  |E_{y,\lambda/a,f}|^{p_2/p_1}dy
  +
 \int_{\bbR^m}| E_{y,a,g}|^{q_2/q_1}dy\right) .
\end{align*}
Observe that
\begin{align*}
\int_{\bbR^m}  |E_{y,\lambda/a,f}|^{p_2/p_1}dy
 &= \left(\frac{a}{\lambda}\right)^{p_2}
 \left(
\frac{\lambda}{a}\left(
\int_{\bbR^m}  |E_{y,\lambda/a,f}|^{p_2/p_1}dy\right)^{1/p_2}\right)^{p_2}   \\
&\le \left(\frac{a}{\lambda}\right)^{p_2} \|f\|_{L^{\vec p, \infty}}^{p_2}.
\end{align*}
Similarly we get that
\[
   \int_{\bbR^m}| E_{y,a,g}|^{q_2/q_1}dy
  \le \frac{1}{a^{q_2}}\|g\|_{L^{\vec q, \infty}}^{q_2}.
\]
Hence
\begin{align}
  \int_{\bbR^m} |E_{y,\lambda,fg}|^{r_2/r_1} dy
  &\le C_{p_2/p_1}  \left( \left(\frac{a}{\lambda}\right)^{p_2}\|f\|_{L^{\vec p, \infty}}^{p_2}
    + \frac{\|g\|_{L^{\vec q, \infty}}^{q_2}}{a^{q_2}}\right). \label{eq:H:e1}
\end{align}
Take some $a_0>0$ such that
\[
  \frac{(a_0/\lambda)^{p_2}\|f\|_{L^{\vec p, \infty}}^{p_2}}{r_2/p_2}
  = \frac{(1/a_0)^{q_2}\|g\|_{L^{\vec q, \infty}}^{q_2}}{r_2/q_2}.
\]
We see from Lemma~\ref{Lm: inequalities} (ii) that
\[
 \left(\frac{a_0}{\lambda}\right)^{p_2}\|f\|_{L^{\vec p, \infty}}^{p_2}
    + \frac{\|g\|_{L^{\vec q, \infty}}^{q_2}}{a_0^{q_2}}
    =\left(\frac{p_2}{r_2}\right)^{r_2/p_2}
  \left(\frac{q_2}{r_2}\right)^{r_2/q_2}
  \cdot \frac{1}{\lambda^{r_2}}\|f\|_{L^{\vec p, \infty}}^{r_2}\|g\|_{L^{\vec q, \infty}}^{r_2}.
\]
By (\ref{eq:H:e1}), we have
\begin{align*}
\lambda \left(\int_{\bbR^m} |E_{y,\lambda,fg}|^{r_2/r_1} dy\right)^{1/r_2}
    &\le  C_{p_2/p_1}^{1/r_2}\left(\frac{p_2}{r_2}\right)^{1/p_2}
  \left(\frac{q_2}{r_2}\right)^{1/q_2}\|f\|_{L^{\vec p, \infty}}\|g\|_{L^{\vec q, \infty}}.
\end{align*}
Taking the supremum over $\lambda$, we get the conclusion as desired.

Next we consider the case of $p_2=q_2=\infty$. In this case, $r_2=\infty$.
If $p_1=\infty$ or $q_1=\infty$, the conclusion is obvious.
For the case of $0<p_1,q_1<\infty$,
we see from (\ref{eq:H:e2}) that
\[
|E_{y,\lambda,fg}|^{1/r_1} \le C_{1/r_1} \left(
 |E_{y,\lambda/a,f}|^{1/r_1}+ | E_{y,a,g}|^{1/r_1}\right).
\]
Hence
\begin{align*}
\left\||E_{y,\lambda,fg}|^{1/r_1}\right\|_{L_y^{\infty}}
&\le C_{1/r_1} \left(
 \left\| |E_{y,\lambda/a,f}|^{1/r_1}\right\|_{L_y^{\infty}}+
 \left\| | E_{y,a,g}|^{1/r_1}\right\|_{L_y^{\infty}}\right)\\
&= C_{1/r_1} \left(
 \left\| |E_{y,\lambda/a,f}|^{1/p_1}\right\|_{L_y^{\infty}}^{p_1/r_1}+
 \left\| | E_{y,a,g}|^{1/q_1}\right\|_{L_y^{\infty}}^{q_1/r_1}\right).
\end{align*}
Since
\begin{align*}
\left\| |E_{y,\lambda/a,f}|^{1/p_1}\right\|_{L_y^{\infty}}^{p_1/r_1}
&= \left(\frac{a}{\lambda} \right)^{p_1/r_1}
 \left\| \frac{\lambda}{a}|E_{y,\lambda/a,f}|^{1/p_1}\right\|_{L_y^{\infty}}^{p_1/r_1}\\
&\le   \left(\frac{a}{\lambda} \right)^{p_1/r_1}
     \|f\|_{L^{\vec p,\infty}}^{p_1/r_1}
\end{align*}
and
\[
  \left\| | E_{y,a,g}|^{1/q_1}\right\|_{L_y^{\infty}}^{q_1/r_1}
  \le \frac{1}{a^{q_1/r_1}}
     \|g\|_{L^{\vec q,\infty}}^{q_1/r_1},
\]
with similar arguments as the previous case we get
\[
  \|fg\|_{L^{\vec r,\infty}} \le C_{1/r_1}
            \left(\frac{p_1}{r_1}\right)^{r_1/p_1}
  \left(\frac{q_1}{r_1}\right)^{r_1/q_1} \|f\|_{L^{\vec p,\infty}} \|g\|_{L^{\vec q, \infty}}.
\]

For $p_1=q_1=\infty$, we see from (\ref{eq:H:fg}) that
\[
  \|\chi^{}_{E_{y,\lambda,fg}}\|_{L_x^{\infty}}
  \le \|\chi^{}_{E_{y,\lambda/a,f}}\|_{L_x^{\infty}}
   + \|\chi^{}_{E_{y,a,g}}\|_{L_x^{\infty}}.
\]
Since every term in the above inequality is either $0$ or $1$,
we have
\[
  \|\chi^{}_{E_{y,\lambda,fg}}\|_{L_x^{\infty}}^{r_2}
  \le \|\chi^{}_{E_{y,\lambda/a,f}}\|_{L_x^{\infty}}^{p_2}
   + \|\chi^{}_{E_{y,a,g}}\|_{L_x^{\infty}}^{q_2}.
\]
Hence
\begin{align}
\int_{\bbR^m}\|\chi^{}_{E_{y,\lambda,fg}}\|_{L_x^{\infty}}^{r_2} dy
&\le
   \int_{\bbR^m}
      \|\chi^{}_{E_{y,\lambda/a,f}}\|_{L_x^{\infty}}^{p_2} dy
     + \int_{\bbR^m}  \|\chi^{}_{E_{y,a,g}}\|_{L_x^{\infty}}^{q_2} dy \nonumber \\
&=
   \left(\frac{a}{\lambda}\right)^{p_2}
    \int_{\bbR^m}
     \left  \| \frac{\lambda}{a} \chi^{}_{E_{y,\lambda/a,f}}\right \|_{L_x^{\infty}}^{p_2} dy  \nonumber \\
 &\qquad + \frac{1}{ a^{q_2}} \cdot
  \int_{\bbR^m}  \left\|a\chi^{}_{E_{y,a,g}}  \right\|_{L_x^{\infty}}^{q_2} dy  \nonumber \\
&\le
   \left(\frac{a}{\lambda}\right)^{p_2} \|f\|_{L^{\vec p,\infty}}^{p_2}
   +
      \frac{1}{ a^{q_2}}\|g\|_{L^{\vec q,\infty}}^{q_2}. \label{eq:H:e4}
\end{align}
Take some $a>0$ such that
\[
    \frac{(a/\lambda)^{p_2} \|f\|_{L^{\vec p,\infty}}^{p_2} }{r_2/p_2}
    =\frac{(1/a)^{q_2}  \|g\|_{L^{\vec q,\infty}}^{q_2}}{r_2/q_2}.
\]
We see from Lemma~\ref{Lm: inequalities} (ii) and (\ref{eq:H:e4}) that
\[
\int_{\bbR^m}\|\chi^{}_{E_{y,\lambda,fg}}\|_{L_x^{\infty}}^{r_2} dy
\le  \left(\frac{p_2}{r_2}\right)^{r_2/p_2}
  \left(\frac{q_2}{r_2}\right)^{r_2/q_2}
  \frac{1}{\lambda^{r_2}}\|f\|_{L^{\vec p, \infty}}^{r_2}\|g\|_{L^{\vec q, \infty}}^{r_2}.
\]
Hence
\[
\lambda \left( \int_{\bbR^m}\|\chi^{}_{E_{y,\lambda,fg}}\|_{L_x^{\infty}}^{r_2} dy\right)^{1/r_2}
\le  \left(\frac{p_2}{r_2}\right)^{1/p_2}
  \left(\frac{q_2}{r_2}\right)^{1/q_2}
  \|f\|_{L^{\vec p, \infty}} \|g\|_{L^{\vec q, \infty}}.
\]
Therefore,
\[
\|fg\|_{L^{\vec r,\infty}}
\le  \left(\frac{p_2}{r_2}\right)^{1/p_2}
  \left(\frac{q_2}{r_2}\right)^{1/q_2}
  \|f\|_{L^{\vec p, \infty}} \|g\|_{L^{\vec q, \infty}}.
\]

Finally, the case of $\vec p =(\infty, \infty)$ or $\vec q=(\infty, \infty)$ is obvious.
 This completes the proof.
\end{proof}

The following examples show the necessity part in Theorem~\ref{thm: holder mixed weak norm}.
In other words,  whenever $p_1q_2\ne p_2q_1$,
H\"older's inequality does not hold.

\begin{Example} \label{Ex: holder}
Suppose that $1/r_i  = 1/p_i + 1/q_i$, $i=1,2$.

\begin{enumerate}

\item For $0<q_1, q_2\le \infty$ and $g(x,y) = (|x|^n+|y|^m)^{-\gamma}$ with $\gamma = 1/q_1 + 1/q_2$,
   we have $g\in L^{\vec q,\infty}$.

\item For $q_1 = \infty$, $0<p_2\le \infty$ and $0<p_1,q_2<\infty$, set
   $\gamma = 1/q_2$ and $\alpha = p_1/p_2+p_1/q_2$.
Let $f(x,y)= (|x|^n+|y|^m)^{\gamma} \chi^{}_{E}(x,y)$ and $g(x,y) = (|x|^n+|y|^m)^{-\gamma}$,
where $E=\{(x,y):\, 0<|x|^n<|y|^{-m\alpha}, 1\le |y|\le N\}$.
Then we have
\[
 \lim_{N\rightarrow\infty}  \frac{ \| fg\|_{L^{\vec r, \infty}}}{\|f\|_{L^{\vec p,\infty}} \|g\|_{L^{\vec q,\infty}}}
   = \infty.
\]

\item
For $q_2=\infty$, $0<p_1\le \infty$ and $0<p_2,q_1<\infty$, set $\gamma = n/q_1$.
Let $f(x,y)= |x|^{\gamma} \chi^{}_{E}(x,y)$ and $g(x,y) = |x|^{-\gamma}$,
where $E=\{(x,y):\,  |x|^n \le |y|^{-mr_1/r_2} \}$.
Then we have
$$\| fg\|_{L^{\vec r, \infty}} \not \lesssim  \|f\|_{L^{\vec p,\infty}} \|g\|_{L^{\vec q,\infty}}.$$

\item \label{item:a}
For $0<p_1,p_2,q_1, q_2<\infty$ with $p_2/q_2 > p_1/q_1$,
set
\[
  \frac{\alpha}{m} = \frac{1}{q_2} - \frac{\beta}{q_1} ,
  \qquad
  \beta = \frac{1/p_2+1/q_2}{1/p_1+1/q_1},
\]
 $f(x,y)= |y|^{\alpha} \chi^{}_{E}(x,y)$ and $g(x,y) = |y|^{-\alpha}\chi^{}_{E}(x,y)$,
where $E=\{(x,y):\, |x|^n \leq |y|^{-m\beta} \}$.
  Then we have
$$\| fg\|_{L^{\vec r, \infty}} \not \lesssim  \|f\|_{L^{\vec p,\infty}} \|g\|_{L^{\vec q,\infty}}.$$

\item
For $0<p_1,p_2,q_1, q_2<\infty$ with $p_2/q_2 < p_1/q_1$,
set
\[
  \frac{\alpha}{m} = \frac{1}{p_2} - \frac{\beta}{p_1},
  \qquad
  \beta = \frac{1/p_2+1/q_2}{1/p_1+1/q_1},
\]
 $f(x,y)= |y|^{-\alpha} \chi^{}_{E}(x,y)$ and $g(x,y) = |y|^{\alpha}\chi^{}_{E}(x,y)$,
where $E=\{(x,y):\, |x|^n \leq |y|^{-m\beta} \}$.
  Then we have
$$\| fg\|_{L^{\vec r, \infty}} \not \lesssim  \|f\|_{L^{\vec p,\infty}} \|g\|_{L^{\vec q,\infty}}.$$
\end{enumerate}
\end{Example}

\begin{proof}
(i).\,\,
   First, we assume that $0<q_1, q_2<\infty$. Denote $E_1=\big\{y:\, |y|^m< \lambda ^{-1/\gamma }\big\}$.
For any $y\in E_1$,
\begin{equation}\label{eq:y in E1}
\left| \left\{x:\, \frac{1}{(|x|^n+|y|^m)^\gamma} >\lambda \right\} \right| \lesssim_{n}  \frac{1}{\lambda^{1/\gamma}}.
\end{equation}
And for $y\not\in E_1$,
\begin{equation}\label{eq:y not in E1}
\left| \left\{x:\, \frac{1}{(|x|^n+|y|^m)^\gamma} >\lambda \right\} \right| =0.
\end{equation}
So
\begin{align}
& \hskip -20mm\left(\int_{\mathbb{R}^{m}} \left| \left\{x:\, \frac{1}{(|x|^n+|y|^m)^\gamma} >\lambda \right\} \right|^{q_2/q_1} dy \right)^{1/q_2}
 \nonumber  \\
&\leq C_{n} \left( \int_{E_1}  \frac{1}{\lambda^{ q_2/(\gamma q_1)}}  dy \right)^{1/q_2}  \nonumber \\
&=C_{n,m}\frac{1}{\lambda^{1 /(\gamma q_1) + 1 /(\gamma q_2)}}.  \label{eq:t2}
\end{align}
Hence for $\gamma = 1/q_1 + 1/q_2$, we have $g\in L^{\vec q,\infty}$.
With suitable modification of the above arguments we see that the conclusion is also true if $q_1=\infty$
or $q_2=\infty$.

(ii).\,\, For $0<p_2<\infty$, a simple computation shows that
\begin{align*}
\|fg\|_{L^{\vec r, \infty}}
&=  \|\chi^{}_{E} \|_{L^{\vec r}}
   \approx \Big (\int_{1\le |y|\le N}  |y|^{-m \alpha r_2/r_1} dy\Big)^{1/r_2} \\
&\approx  \Big (\int_1^N  t^{-m\alpha r_2/r_1+m-1} dt\Big)^{1/r_2}
\end{align*}
and
\begin{align*}
\|f\|_{L^{\vec p}}
  &\le \Big( \int_{1\le |y|\le N}  |y|^{m\gamma  p_2 -m \alpha p_2/p_1} dy\Big)^{1/p_2} \\
  &\lesssim \Big( \int_1^N   t^{m\gamma p_2 -m \alpha p_2/p_1+m-1} dy\Big)^{1/p_2}.
\end{align*}
Since $\alpha = p_1/p_2 + p_1/q_2$ and $p_1=r_1$, we have $\alpha r_2/r_1 = 1$.
Moreover,
\[
 \gamma p_2 - \frac{\alpha p_2}{p_1} + 1
   = \frac{p_2}{q_2} - \frac{p_2}{p_1} \left( \frac{p_1}{p_2} + \frac{p_1}{q_2}\right) + 1 = 0.
\]
Hence
\[
\|fg\|_{L^{\vec r, \infty}}\approx (\ln N)^{1/r_2}
\mathrm{\,\,\,\, and\,\,\,\,}
  \|f\|_{L^{\vec p}}\lesssim (\ln N)^{1/p_2}.
\]

And for $p_2=\infty$, it is easy to see that
\[
\|fg\|_{L^{\vec r, \infty}}\approx (\ln N)^{1/r_2}
\mathrm{\,\,\,\, and\,\,\,\,}
  \|f\|_{L^{\vec p}}\lesssim 1.
\]

By (i), $\|g\|_{L^{\vec q, \infty}}<\infty$.
Since $p_2>r_2$ and
\[
  \frac{ \| fg\|_{L^{\vec r, \infty}}}{\|f\|_{L^{\vec p,\infty}} \|g\|_{L^{\vec q,\infty}}}
  \ge
  \frac{ \| fg\|_{L^{\vec r, \infty}}}{\|f\|_{L^{\vec p}} \|g\|_{L^{\vec q,\infty}}}
 \gtrsim (\ln N)^{1/r_2 - 1/p_2},
\]
 we get the conclusion as desired.

(iii).\,\, It is easy to see
$$\|fg\|_{L^{\vec r, \infty}}
=  \|\chi^{}_{E} \|_{L^{\vec r}}
\approx  \Big (\int_{\bbR^m}  |y|^{-m} dy\Big)^{1/r_2} =\infty.$$
On the other hand, for $0<p_1<\infty$,
\begin{align*}
\|f\|_{L^{\vec p, \infty}}
&=\sup\limits_{\lambda>0} \lambda \left(\int_{\bbR^m}
  \left|\{x:\, |f(x,y)| >\lambda\}\right|^{p_2/p_1} dy\right)^{1/p_2} \\
&= \sup\limits_{\lambda>0} \lambda \left(\int_{\bbR^m} \left|\{x:\,  \lambda^{n/\gamma}< |x|^n
  \leq  |y|^{-mr_1/  r_2 } \}\right|^{p_2/p_1} dy\right)^{1/p_2} \\
&\lesssim \sup\limits_{\lambda>0} \lambda \left(\int_{|y|<\lambda^{-nr_2/(m\gamma r_1)}}
(|y|^{-mr_1/  r_2 } -\lambda^{n/\gamma})^{p_2/p_1} dy\right)^{1/p_2} \\
&\lesssim\sup\limits_{\lambda>0} \lambda \left(\int_{|y|<\lambda^{-nr_2/(m\gamma r_1)}}
(|y|^{-mr_1/  r_2 }  )^{p_2/p_1} dy\right)^{1/p_2}   \lesssim 1.
\end{align*}
And for $p_1=\infty$,
we also have $\|f\|_{L^{\vec p, \infty}}\lesssim 1$.

Observe that
\begin{align*}
\|g\|_{L^{\vec q,\infty}}
&=\sup\limits_{\lambda>0} \lambda \left\|  |\{x:\,  |x|^{-\gamma} > \lambda\}|^{1/q_1}\right\|_{L_{y}^{\infty}} \\
&=\sup\limits_{\lambda>0} \lambda \left\| |\{x:\,   |x|< \lambda^{-1/\gamma} \}|^{1/q_1} \right\|_{L_{y}^{\infty}}  \\
&\lesssim  \sup\limits_{\lambda>0} \lambda (\lambda^{-n/\gamma})^{1/q_1} =1.
\end{align*}
We get the conclusion as desired.

(iv).\,\, We see from the definition of $\alpha$ and $\beta $ that
\[
  \frac{1}{q_2} = \frac{\alpha}{m}+ \frac{\beta}{q_1} ,
\quad
 \frac{\beta}{p_1} =\frac{\alpha}{m} + \frac{1}{p_2}.
\]
First, we show that $\alpha>0$, which is equivalent to
\[
  \frac{q_1}{q_2} >  \frac{1/p_2+1/q_2}{1/p_1+1/q_1}.
\]
Or equivalently,
\[
  \frac{q_1}{p_1} > \frac{q_2}{p_2},
\]
which is true by the hypothesis.

Next we show that $f\in L^{\vec p,\infty}$. In fact,
\begin{align*}
\|f\|_{L^{\vec p, \infty}}
&=\sup\limits_{\lambda>0} \lambda \left(\int_{\bbR^m} \left|\{x:\,  |y|^{\alpha} \chi^{}_{E}(x,y) >\lambda\}\right|^{p_2/p_1} dy\right)^{1/p_2} \\
&\lesssim \sup\limits_{\lambda>0} \lambda \left(\int_{|y|>\lambda^{1/\alpha}}   |y|^{-m \beta p_2/p_1} dy \right)^{1/p_2} \\
&\lesssim \sup\limits_{\lambda>0} \lambda \left(\int_{\lambda^{1/\alpha}}^{\infty}   t^{-m \beta p_2/p_1+m-1} dt \right)^{1/p_2} \\
&= \sup\limits_{\lambda>0} \lambda  \lambda^{-1}  =1.
\end{align*}
Similarly, we prove that $g\in L^{\vec q,\infty}$.
\begin{align*}
\|g\|_{L^{\vec q,\infty}}
&=\sup_{  \lambda>0} \lambda \left(\int_{\bbR^m} \left|\{x:\,  |y|^{-\alpha} \chi^{}_{E}(x,y) >\lambda\}\right|^{q_2/q_1} dy\right)^{1/q_2} \\
&\lesssim \sup_{ \lambda>0}
      \lambda \left(\int_{|y|<\lambda^{-1/\alpha}}  |y|^{-m \beta q_2 / q_1} dy \right)^{1/q_2} \\
&\lesssim \sup_{ \lambda>0}
   \lambda \left(\int_{0}^{\lambda^{-1/\alpha}}  t^{-m \beta q_2/q_1+m-1} dt \right)^{1/q_2}\\
   &\lesssim 1.
\end{align*}
 However,
$$\|fg\|_{L^{\vec r, \infty}}
=  \|\chi^{}_{E} \|_{L^{\vec r}}
= \Big (\int_{\bbR^m}  |y|^{-m\beta r_2/r_1} dy\Big)^{1/r_2} =\infty.$$
Hence $\| fg\|_{L^{\vec r, \infty}} \not \lesssim  \|f\|_{L^{\vec p,\infty}} \|g\|_{L^{\vec q,\infty}}$.

(v).\,\, By an interchange of $f$ and $g$ in (iv), we get (v).
\end{proof}

\subsection{Interpolation}

It is well known that for $p<r<q$, we have $L^p\cap L^q \subset L^r$. The same is true for weak Lebesgue spaces.
Moreover, we have the following interpolation formula.

\begin{Proposition}[{\cite[Propisition 1.1.14]{Grafakos2008}}] \label{prop: interpolation}
Let $p<q\le \infty$ and $f\in L^{p,\infty}\cap L^{q,\infty}$.
Then $f$ is in $L^r$ for all  $p<r<q$ and
\[
  \|f\|_{L^r} \le \left(\frac{r}{r-p} + \frac{r}{q-r}\right)^{1/r}
           \|f\|_{L^{p,\infty}}^{\theta} \|f\|_{L^{q,\infty}}^{1-\theta}
\]
with the suitable interpretation when $q=\infty$,
 where $0<\theta<1$ satisfies
  $1/r = \theta /p + (1-\theta)/q$.
\end{Proposition}

However, the above proposition is not true in general if $p, q, r$ are replaced with vector indices.

\begin{Theorem} \label{thm:interpolation}
Suppose that $\vec p=(p_1, p_2)$, $\vec q=(q_1, q_2)$ and $\vec r=(r_1, r_2)$ satisfy that
\begin{equation}\label{eq:rpq}
  \frac{1}{r_i} = \frac{\theta}{p_i} + \frac{1-\theta}{q_i},\qquad i=1,2,
\end{equation}
where $0<\theta<1$ is a constant.
Then we have
\begin{align*}
\|f\|_{L^{\vec r, \infty}} &\le \|f\|_{L^{\vec p, \infty}}^{\theta}\|f\|_{L^{\vec q, \infty}}^{1-\theta}, \\
 \|f\|_{L^{\vec r}}
 &\le
   \left(\frac{r_1}{r_1-p_1} + \frac{r_1}{q_1-r_1}\right)^{1/r_1}
     \|f\|_{L^{p_2}(L^{p_1,\infty})}^{\theta}
      \|f \|_{L^{q_2}(L^{q_1,\infty})}^{1-\theta}.
\end{align*}

However,
if   $1/p_1 + 1/p_2 = 1/q_1 + 1/q_2$,
then for any multiple index $\vec r$,
$L^{\vec p, \infty}\cap L^{\vec q,\infty}  \not\subset L^{\vec r}$.
\end{Theorem}

\begin{proof}
Fix some function $f$ and $\lambda>0$. We see from H\"older's inequality that
\begin{align*}
\lambda \|\chi^{}_{\{|f|>\lambda\}} \|_{L^{\vec r}}
 &\le \lambda^{\theta} \|\chi^{}_{\{|f|>\lambda\}} \|_{L^{\vec p/\theta}}
 \cdot
 \lambda^{1-\theta} \|\chi^{}_{\{|f|>\lambda\}} \|_{L^{\vec q/(1-\theta)}} \\
 &\le \|f\|_{L^{\vec p,\infty}}^{\theta} \|f\|_{L^{\vec q,\infty}}^{1-\theta}.
\end{align*}
Hence
$    \|f\|_{L^{\vec r, \infty}} \le \|f\|_{L^{\vec p, \infty}}^{\theta}\|f\|_{L^{\vec q, \infty}}^{1-\theta}$.

On the other hand,
we see from Proposition~\ref{prop: interpolation} that
\[
  \|f(\cdot,y)\|_{L_x^{r_1}} \le \left(\frac{r_1}{r_1-p_1} + \frac{r_1}{q_1-r_1}\right)^{1/r_1}
        \|f(\cdot,y)\|_{L_x^{p_1,\infty}}^{\theta}  \|f(\cdot,y)\|_{L_x^{q_1,\infty}}^{1-\theta}.
\]
By H\"older's inequality, we get
\begin{align*}
 \|f\|_{L^{\vec r}}
 &\le \left(\frac{r_1}{r_1-p_1} + \frac{r_1}{q_1-r_1}\right)^{1/r_1}
     \left\|\|f\|_{L_x^{p_1,\infty}}^{\theta}\right\|_{L_y^{p_2/\theta}}
     \left\|\|f(\cdot,y)\|_{L_x^{q_1,\infty}}^{1-\theta}\right\|_{L_y^{q_2/(1-\theta)}} \\
 &= \left(\frac{r_1}{r_1-p_1} + \frac{r_1}{q_1-r_1}\right)^{1/r_1}
     \|f\|_{L^{p_2}(L^{p_1,\infty})}^{\theta}
      \|f \|_{L^{q_2}(L^{q_1,\infty})}^{1-\theta}.
\end{align*}

Next we show that $L^{\vec p, \infty}\cap L^{\vec q,\infty}  \not\subset L^{\vec r}$.
Set $f(x,y) = (|x|^n+|y|^m)^{-\gamma}$, where $\gamma>0$.
For $\gamma = 1/q_1 + 1/q_2=1/p_1+1/p_2$, we see from
Example~\ref{Ex: holder}(i) that
  $f\in L^{\vec p,\infty}\cap L^{\vec q,\infty}$.

It remains to show that $f\not\in L^{\vec r}$.
We have
\begin{align*}
\int_{\bbR^n} |f(x,y)|^{r_1}dx
&\ge \int_{|x|^n<|y|^m} |f(x,y)|^{r_1}dx +  \int_{|x|^n\ge |y|^m} |f(x,y)|^{r_1}dx \\
&\ge \int_{|x|^n<|y|^m} \frac{1}{2^{\gamma r_1}|y|^{m\gamma r_1}}dx +  \int_{|x|^n\ge |y|^m} \frac{1}{2^{\gamma r_1}|x|^{n\gamma r_1}}dx.
\end{align*}
Hence for $\gamma r_1 \le 1$, we have $\|f(\cdot,y)\|_{L^{r_1}_x} = \infty$. And for $\gamma r_1>1$, we have
\[
  \int_{\bbR^n} |f(x,y)|^{r_1}dx  \ge C_{n,r_1} \frac{1}{|y|^{m\gamma r_1 - m}}.
\]
In both cases, we have $\|f\|_{L^{\vec r}} = \infty$.
\end{proof}


When the iterated weak norms are invoked,
we get again an interpolation theorem.
However, four iterated weak norms are invoked.
\begin{Theorem}
Suppose that
\begin{align*}
  \frac{1}{r_1} &= \frac{\theta}{p_1} + \frac{1-\theta}{q_1}, \\
  \frac{1}{r_2} &=
  \frac{\theta\xi}{p_{21}} + \frac{(1-\theta)\xi}{p_{22}}
    + \frac{\theta(1-\xi)}{q_{21}} + \frac{(1-\theta)(1-\xi)}{q_{22}},
\end{align*}
where $0<\theta,\xi<1$ are constants.
Then we have
\[
    \|f\|_{L^{\vec r}} \le   C\|f  \|_{L^{p_{21},\infty}(L^{p_1,\infty})}^{\theta \xi}
      \|f    \|_{L^{p_{22},\infty}(L^{q_1,\infty})}^{(1-\theta) \xi}
     \|f   \|_{L^{q_{21},\infty}(L^{p_1,\infty})}^{\theta (1-\xi)}
     \|f   \|_{L^{q_{22},\infty}(L^{q_1,\infty})}^{(1-\theta) (1-\xi)}.
\]
\end{Theorem}

\begin{proof}
We see from Proposition~\ref{prop: interpolation} that
\[
  \|f(\cdot,y)\|_{L_x^{r_1}} \le \left(\frac{r_1}{r_1-p_1} + \frac{r_1}{q_1-r_1}\right)^{1/r_1}
        \|f(\cdot,y)\|_{L_x^{p_1,\infty}}^{\theta}  \|f(\cdot,y)\|_{L_x^{q_1,\infty}}^{1-\theta}.
\]
Set
\[
  \frac{1}{p_2} = \frac{\theta}{p_{21}} + \frac{1-\theta}{p_{22}}
  \quad \mathrm{and}\quad
  \frac{1}{q_2} = \frac{\theta}{q_{21}} + \frac{1-\theta}{q_{22}}.
\]
Then we have  $1/r_2 = \xi/p_2 + (1-\xi)/q_2$. Using Proposition~\ref{prop: interpolation} again we get
\begin{align}
 \|f\|_{L^{\vec r}}
 &\le C \left\|\|f(\cdot,y)\|_{L_x^{p_1,\infty}}^{\theta}  \|f(\cdot,y)\|_{L_x^{q_1,\infty}}^{1-\theta}
     \right\|_{L_y^{p_2,\infty}}^{\xi} \nonumber \\
    &\qquad \times
     \left\|\|f(\cdot,y)\|_{L_x^{p_1,\infty}}^{\theta}  \|f(\cdot,y)\|_{L_x^{q_1,\infty}}^{1-\theta}
      \right\|_{L_y^{q_2,\infty}}^{1-\xi}.
             \nonumber
\end{align}
Now we see from H\"older's inequality for weak norms that
\begin{align}
 \|f\|_{L^{\vec r}}
 &\le C \left\|\|f(\cdot,y)\|_{L_x^{p_1,\infty}}^{\theta}\right\|_{L_y^{p_{21}/\theta,\infty}}^{ \xi}
       \left\|  \|f(\cdot,y)\|_{L_x^{q_1,\infty}}^{1-\theta}
     \right\|_{L_y^{p_{22}/(1-\theta),\infty}}^{\xi} \nonumber \\
    &\qquad \times
      \left\|\|f(\cdot,y)\|_{L_x^{p_1,\infty}}^{\theta}\right\|_{L_y^{q_{21}/\theta,\infty}}^{ 1-\xi}
       \left\|  \|f(\cdot,y)\|_{L_x^{q_1,\infty}}^{1-\theta}
     \right\|_{L_y^{q_{22}/(1-\theta),\infty}}^{1-\xi}
             \nonumber \\
&= C\|f  \|_{L^{p_{21},\infty}(L^{p_1,\infty})}^{\theta \xi}
      \|f  \|_{L^{p_{22},\infty}(L^{q_1,\infty})}^{(1-\theta) \xi}
     \|f  \|_{L^{q_{21},\infty}(L^{p_1,\infty})}^{\theta (1-\xi)}
     \|f   \|_{L^{q_{22},\infty}(L^{q_1,\infty})}^{(1-\theta) (1-\xi)}.
\end{align}
\end{proof}

The above results can be restated as follows.
Denote  $\vec p_a = (p_1, p_{21})$,
 $\vec p_b = (p_1, q_{21})$,
 $\vec q_a = (q_1, p_{22})$ and
 $\vec q_b = (q_1, q_{22})$.
Let
 $(1/r_1, 1/r_2)$ be a point in the interior of the quadrilateral determined by
the four points
 $(1/p_1, 1/p_{21})$,
 $(1/p_1, 1/q_{21})$,
 $(1/q_1, 1/p_{22})$,
 $(1/q_1, 1/q_{22})$
 as shown  in Figure~\ref{Fig:fig1}.
Then we have
\[
  L^{\vec p_a,\infty}\cap L^{\vec p_b,\infty}\cap L^{\vec q_a,\infty}\cap L^{\vec q_b,\infty}
  \subset L^{\vec r}.
\]

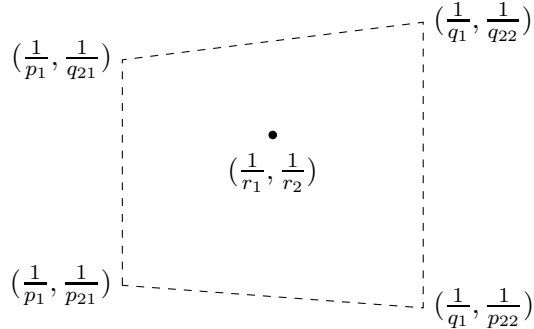
\begin{figure}[!ht]

\begin{center}
\begin{tikzpicture}
\coordinate  [label=left:${( \frac{1}{p_1}, \frac{1}{p_{21}})}$](Pa) at (0,0);
\coordinate  [label=left:${( \frac{1}{p_1}, \frac{1}{q_{21}})}$](Pb) at (0,3);
\coordinate  [label=right:${(\frac{1}{q_1}, \frac{1}{p_{22}})}$](Qa) at (4,-0.3);
\coordinate  [label=right:${(\frac{1}{q_1}, \frac{1}{q_{22}})}$] (Qb) at (4,3.5);

\filldraw (2,2) circle  (0.05);
\draw (2,1.5) node {$(\frac{1}{r_1}, \frac{1}{r_2})$};
\draw[dashed] (Pa) -- (Pb) -- (Qb) -- (Qa) -- (Pa);
\end{tikzpicture}

\caption{Interpolation area} \label{Fig:fig1}
\end{center}
\end{figure}

\section{Geometric Inequalities Related to Fractional Integration}

In this section, we study
the boundedness of $T_{\gamma}$ and $L_{\gamma}$ from $L^{\vec p}$ to $L^{\vec q}$.
First, we consider $T_{\gamma}$ with $\vec p = (\infty, \infty)$.  In this case, it is more convenient to rewrite the inequality in the following
form,
\[
  \|F\|_{X}  \lesssim \sup\limits_{x,y \in \mathbb{R}^{n}} F(x,y) (|x+y|+|x-y|)^{ n/q_1 + n/q_2  },
\]
where $X$ stands for some norm defined on $\bbR^{2n}$. Recall that $L^{\vec p}= L^{\infty}$ whenever $\vec p = (\infty, \infty)$.

   First, we point out that the following inequality
   \begin{equation}\label{eq: false sup inequality}
   \|F\|_{L^{\vec p, \infty}}  \leq C_{\vec p,n}  \sup\limits_{x,y \in \mathbb{R}^{n}} F(x,y) |x-y|^{ n/p_1 + n/p_2  }
   \end{equation}
   if false,  even for characteristic functions. This is because for any $s>0$,
   $$|\{(x, y):\, |x-y|<s \}|=\infty.$$
   Now we  turn  to study its fractional form as follows.

\begin{Theorem} \label{thm:weak norm}
    Let $F$ be a nonnegative measurable function  defined on $\mathbb{R}^{2n}$. Then for all $0<q_1, q_2 \leq \infty$, we have
    \begin{equation}\label{eq:weak norm}
    \|F\|_{L^{\vec q, \infty}}  \leq C_{\vec q,n}  \sup\limits_{x,y \in \mathbb{R}^{n}} F(x,y) (|x+y|+|x-y|)^{ n/q_1 + n/q_2  },
    \end{equation}
    \begin{equation}\label{eq:weak weak norm}
    \|F\|_{L^{q_2,\infty}(L^{q_1,\infty})}  \leq C_{\vec q,n}  \sup\limits_{x,y \in \mathbb{R}^{n}} F(x,y) (|x+y|+|x-y|)^{ n/q_1 + n/q_2  }.
    \end{equation}
    However, for $\vec q \ne (\infty, \infty)$, we have
    \begin{equation}\label{eq: norm}
    \|F\|_{L^{\vec q}}  \leq C_{\vec q,n}  \sup\limits_{x,y \in \mathbb{R}^{n}} F(x,y) (|x+y|+|x-y|)^{ n/q_1 + n/q_2  }
    \end{equation}
    is not true for all $F\in L^{\vec q}(\mathbb{R}^{2n})$.
    \end{Theorem}

    \begin{proof}
    (i). First, we prove (\ref{eq:weak norm}).
    For $0<q_1, q_2<\infty$, suppose that
    \[
      \sup\limits_{x,y \in \mathbb{R}^{n}} F(x,y) (|x+y|+|x-y|)^{ n/q_1 + n/q_2  }=s<\infty.
    \]
    Denote $E_{y,\lambda } = \{x:\, |F(x,y)|>\lambda\}$.
    For any $\lambda >0$,
    $$\Big\| |E_{y,\lambda}|^{1/q_1} \Big\|_{L^{q_2}}
    \leq \left(\int_{\mathbb{R}^{n}} \!\left| \left\{x\!: (|x+y|+|x-y|)^{ n/q_1 + n/q_2  } < \frac{s}{\lambda} \right\} \right|^{q_2/q_1} \! dy \!\right)^{1/q_2}.$$
    Denote $E_1=\big\{y:\, |y|< (1/2)(s/\lambda)^{ q_1q_2/(nq_1+nq_2)}\big\}$.
    For any $y\in E_1$,
    \begin{equation}\label{eq:y in E1 b}
    \left| \left\{x:\, (|x+y|+|x-y|)^{ n/q_1 + n/q_2  } < \frac{s}{\lambda} \right\} \right| \lesssim_{n}  \Big(\frac{s}{\lambda}\Big)^{ q_1q_2/(q_1+q_2)}.
    \end{equation}
    And for $y\not\in E_{1}$,
    \begin{equation}\label{eq:y not in E1 b}
    \left| \left\{x: (|x+y|+|x-y|)^{ n/q_1 + n/q_2  } < \frac{s}{\lambda} \right\} \right| =0.
    \end{equation}
    So
    \begin{align*}
    & \hskip -20mm\left(\int_{\mathbb{R}^{n}} \left| \left\{x:\, (|x+y|+|x-y|)^{ n/q_1 + n/q_2  } < \frac{s}{\lambda} \right\} \right|^{q_2/q_1} dy \right)^{1/q_2} \\
    &\leq C_{\vec q,n} \left( \int_{E_1}  \Big(\Big(\frac{s}{\lambda}\Big)^{ q_1q_2/(q_1+q_2)} \Big)^{ q_2/q_1} dy \right)^{1/q_2}\\
    &= C_{\vec q,n}\Big(\frac{s}{\lambda}\Big)^{ q_2/(q_1 + q_2)}
      \Big(\int_{E_1} dy \Big)^{1/q_2}\\
    &=C_{\vec q,n}\frac{s}{\lambda}.
    \end{align*}
    Therefore,
    $$\|F\|_{L^{\vec q, \infty}}= \sup_{\lambda>0} \lambda \Big\| |E_{y,\lambda}|^{1/q_1} \Big\|_{L^{q_2}}
    \lesssim_{n}    \sup\limits_{x,y \in \mathbb{R}^{n}} F(x,y) (|x + y|+|x-y|)^{ n/q_1 + n/q_2  }.$$

    Now we consider the case of endpoints. Since the case of $q_1=q_2=\infty$ is obvious, we assume  that one of $q_1$ and $q_2$ is finite.

    For  $q_2=\infty$, we have
    $$\|F\|_{L^{\vec q,\infty}}= \sup_{\lambda>0} \lambda \Big\| |E_{y,\lambda}|^{1/q_1} \Big\|_{L^{\infty}}.$$
    Taking similar calculation gives for any $\lambda>0$,
    $$\Big\| |E_{y,\lambda}|^{1/q_1} \Big\|_{L^{\infty}} \leq  \Big\|  \left| \left\{x:\, (|x+y|+|x-y|)^{ n/q_1 }  < \frac{s}{\lambda} \right\} \right|^{1/q_1} \Big\|_{L^{\infty}}
    \leq C_{q_1,n}  \frac{s}{\lambda}.$$
    Hence
    $$\|F\|_{L^{\vec q,\infty}}= \sup_{\lambda>0} \lambda \Big\| |E_{y,\lambda}|^{1/q_1} \Big\|_{L^{\infty}} \leq C_{q_1,n} s.$$

    For $q_1=\infty$, we have
    $$\|F\|_{L^{\vec q,\infty}}= \sup_{\lambda>0} \lambda \Big\|  \| \chi^{}_{E_{y,\lambda}} \|_{L^{\infty}} \Big\|_{L^{q_2}}.$$
    Note that for any $\lambda>0$,
    $$\chi^{}_{E_{y,\lambda}} \leq \chi_{ \{x:\, (|x+y|+|x-y|)^{ n/q_2  }  <  {s}/{\lambda} \} }.$$
    If $|y|< (1/2)( {s}/{\lambda})^{q_2/n}$,
    $\| \chi^{}_{E_{y,\lambda}} \|_{L^{\infty}}\leq 1$;
    if $|y| \geq (1/2)( {s}/{\lambda})^{q_2/n}$,
    $\| \chi^{}_{E_{y,\lambda}} \|_{L^{\infty}}=0$.
    Hence
    $$\Big\|  \| \chi^{}_{E_{y,\lambda}} \|_{L^{\infty}}  \Big\|_{L^{q_2}} \leq C_{q_2,n} \frac{s}{\lambda}.$$
    Therefore,
    $$\|F\|_{L^{\vec q,\infty}}= \sup_{\lambda>0} \lambda \Big\|  \| \chi^{}_{E_{y,\lambda}} \|_{L^{\infty}} \Big\|_{L^{q_2}} \leq C_{q_2,n} s.$$

    (ii). Next we prove (\ref{eq:weak weak norm}).
    As in the previous arguments, we consider first the case of $q_1, q_2<\infty$.

    As in the proof of  (\ref{eq:weak norm}),
    suppose that $\sup_{x,y \in \mathbb{R}^{n}} F(x,y) (|x+y|+|x-y|)^{ n/q_1 + n/q_2  }=s<\infty$.
    Then we see from (\ref{eq:y in E1 b}) and  (\ref{eq:y not in E1 b})  that for any $y \neq 0$,
    \begin{align*}
    & \hskip -10mm\sup_{\lambda>0}  \lambda | \{x:\, |F(x,y)|>\lambda\}|^{1/q_1} \\
    &\leq  \sup_{\lambda>0}  \lambda \left| \left\{x:\,   (|x+y|+|x-y|)^{ n/q_1 + n/q_2  }<   \frac{s}{\lambda}   \right\} \right|^{1/q_1}\\
    &= C_{\vec q,n} s |y|^{-n/q_2}.
    \end{align*}
    Hence
    \begin{align*}
     \|F\|_{L^{q_2,\infty}(L^{q_1,\infty})} &=   \Big\|  \| F\|_{L_{x}^{q_1,\infty}}  \Big\|_{L_{y}^{q_2,\infty}} \\
     &=\sup_{\beta>0} \beta \left|\left\{y:\, \sup_{\lambda>0}  \lambda | \{x:\, |F(x,y)|>\lambda\}|^{1/q_1} > \beta  \right\}\right|^{1/q_2} \\
    &\leq  \sup_{\beta>0} \beta \left|\left\{y:\,  C_{\vec q,n} s |y|^{-n/q_2} > \beta  \right\}\right|^{1/q_2} \\
    &\leq  \sup_{\beta>0} \beta C_{\vec q,n} s \beta^{-1} = C_{\vec q,n} s.
    \end{align*}

    It remains to see the endpoint cases.
    For  $q_2=\infty$,
    $$ \|F\|_{L^{\infty}(L^{q_1,\infty})} =   \Big\|  \| F\|_{L_{x}^{q_1,\infty}}  \Big\|_{L_{y}^{\infty}}.$$
    For any $\lambda>0$ and $y\in\bbR^n$,
    \begin{align*}
    \lambda | \{x:\, |F(x,y)|>\lambda \}|^{1/q_1}
    &\leq  \lambda \left| \left\{x:\,  (|x+y|+|x-y|)^{ n/q_1 }  < \frac{s}{\lambda}  \right \} \right|^{1/q_1} \\
    &\leq C_{q_1,n} s.
    \end{align*}
    Hence
    $$ \|F\|_{L^{\infty}(L^{q_1,\infty})} =   \Big\|  \| F\|_{L_{x}^{q_1,\infty}}  \Big\|_{L_{y}^{\infty}} \leq C_{q_1,n} s.$$

    For $q_1=\infty$,
    $$ \|F\|_{L^{q_2,\infty}(L^{\infty})} =   \Big\|  \| F\|_{L_{x}^{\infty}}  \Big\|_{L_{y}^{q_2,\infty}}.$$
    Since
    $$F(x,y) \leq s (|x+y|+|x-y|)^{- n/q_2  }  \leq s(2|y|)^{- n/q_2  },$$
    we get for all $y$,
    $$\| F\|_{L_{x}^{\infty}}  \leq C_{q_2,n} s |y|^{- n/q_2  }.$$
    Hence $ \Big\|  \| F\|_{L_{x}^{\infty}}  \Big\|_{L_{y}^{q_2,\infty}} \leq C_{q_2,n} s$.

    Finally,  we use a counterexample to show that (\ref{eq: norm}) is false.
    We consider only the case of $q_1, q_2<\infty$. Other cases can be found in
    Theorem~\ref{thm:half weak norm}.

    Let
    $$F_{N}(x,y)=\left(\frac{1}{|x|+|y|}\right)^{ n/q_1 + n/q_2  } \chi_{\{(x,y):\, 1\leq |y| \leq N, |y| \leq |x|\}}, \ \ N \in \mathbb{N}.$$
    Clearly,
    $$\sup_{x,y \in \mathbb{R}^{n}} F_N(x,y) (|x+y|+|x-y|)^{ n/q_1 + n/q_2  } \leq 2^{ n/q_1 + n/q_2  } .$$
    On the other hand,
    $$\|F_N\|_{L^{\vec q}} = \left(\int_{\{y:\, 1\leq |y| \leq N\}} \left(\int_{ \{x:\, |x|\geq|y| \} } \frac{1}{(|x|+|y|)^{n+nq_1/q_2}} dx \right)^{q_2/q_1} dy \right)^{1/q_2}.$$
    By polar coordinates, we have
    \begin{align*}
    \|F_N\|_{L^{\vec q}}
    &= C \left(\int_{\{y:\, 1\leq |y| \leq N\}} \left(\int_{|y|}^{\infty} \frac{r^{n-1}}{(r+|y|)^{n+nq_1/q_2}} dr \right)^{q_2/q_1} dy \right)^{1/q_2} \\
    &\geq C' \left(\int_{\{y:\, 1\leq |y| \leq N\}} \left(\int_{|y|}^{\infty} \frac{r^{n-1}}{r^{n+nq_1/q_2}} dr \right)^{q_2/q_1} dy \right)^{1/q_2} \\
    &=C_{\vec q,n}  \big(\int_{\{y:\, 1\leq |y| \leq N\}} |y|^{-n} dy \big)^{1/q_2} \\
    &=C_{\vec q,n} (\ln   N )^{1/q_2}.
    \end{align*}
    Letting $N \to \infty$, we get
    $$ \|F_N\|_{L^{\vec q}} \rightarrow \infty,$$
    which implies that (\ref{eq: norm}) does not hold for all $F\in L^{\vec q}(\mathbb{R}^{2n})$.
    \end{proof}

Next we consider the boundedness of $T_{\gamma}$ from $L^{\infty}(\bbR^{2n})$ to $X(\bbR^n)$, where $X$ stands for the mixed norm
$L^{q_2,\infty}(L^{q_1})$ or $L^{q_2}(L^{q_1,\infty})$.

    \begin{Theorem}\label{thm:half weak norm}
    Let $F$ be a nonnegative measurable function  defined on $\mathbb{R}^{2n}$. Then for all $0<q_1, q_2 < \infty$ we have
    \begin{equation}\label{eq: strong weak norm}
    \|F\|_{L^{q_2,\infty}(L^{q_1})}  \leq C_{\vec q,n}  \sup\limits_{x,y \in \mathbb{R}^{n}} F(x,y) (|x+y|+|x-y|)^{ n/q_1 + n/q_2  }.
    \end{equation}
    However,
    \begin{equation}\label{eq: weak strong norm}
    \|F\|_{L^{q_2}(L^{q_1,\infty})}  \leq C_{\vec q,n}  \sup\limits_{x,y \in \mathbb{R}^{n}} F(x,y) (|x+y|+|x-y|)^{ n/q_1 + n/q_2  }
    \end{equation}
    does not hold.

    Meanwhile, we present all the endpoint cases. For any  $C_{\vec q,n}>0$,
    \begin{equation}\label{eq: endpoint strong weak norm}
    \|F\|_{L^{\infty}(L^{q_1})}  \nleq C_{\vec q,n}  \sup\limits_{x,y \in \mathbb{R}^{n}} F(x,y) (|x+y|+|x-y|)^{ n/q_1 },
    \end{equation}

    \begin{equation}\label{eq: endpoint weak strong norm}
    \|F\|_{L^{q_1}(L^{\infty})}  \nleq C_{\vec q,n}  \sup\limits_{x,y \in \mathbb{R}^{n}} F(x,y) (|x+y|+|x-y|)^{ n/q_1 }.
    \end{equation}
    For the remaining  endpoint cases, we have
    \begin{equation}\label{eq: endpoint weak weak norm1}
    \|F\|_{L^{q_1, \infty}(L^{\infty})} \leq C_{\vec q,n}  \sup\limits_{x,y \in \mathbb{R}^{n}} F(x,y) (|x+y|+|x-y|)^{ n/q_1 },
    \end{equation}

    \begin{equation}\label{eq: endpoint weak weak norm2}
    \|F\|_{L^{\infty}(L^{q_1, \infty})}  \leq C_{\vec q,n}  \sup\limits_{x,y \in \mathbb{R}^{n}} F(x,y) (|x+y|+|x-y|)^{ n/q_1 }.
    \end{equation}

    \end{Theorem}

    \begin{proof}
    (1) Suppose that $\sup\limits_{x,y \in \mathbb{R}^{n}} F(x,y) (|x+y|+|x-y|)^{ n/q_1 + n/q_2  }=s<\infty$.
    Then
    $$\Big(\int_{\mathbb{R}^{n}} |F(x,y)|^{q_1} dx\Big)^{1/q_1} \leq s \left(\int_{\mathbb{R}^{n}} \frac{1}{(|x+y|+|x-y|)^{n+ nq_1/q_2}} dx \right)^{1/q_1}.$$
    For any $y$, denote
    $M_{y}=\{x:\, |x|<|y|\}$ and $N_{y}=\{x:\, |x| \geq |y|\}$.
    Note that
    $$\frac{1}{|x+y|+|x-y|} \leq \min\Big\{\frac{1}{2|x|}, \frac{1}{2|y|}\Big\}, \qquad  \forall \ x, y \neq 0.$$
     Then by calculation
    \begin{align*}
    &\ \ \ \ \ \ \ \ s \left(\int_{\mathbb{R}^{n}} \frac{1}{(|x+y|+|x-y|)^{n+nq_1/q_2 }} dx \right)^{1/q_1}\\
    &=  s \left(  \Big(\int_{M_{y}}+ \int_{N_{y}} \Big) \frac{1}{(|x+y|+|x-y|)^{n+nq_1/q_2 }} dx  \right)^{1/q_1} \\
    &\leq  s \left(\int_{M_{y}}   \Big(\frac{1}{2|y|}\Big)^{n+nq_1/q_2 }  dx +   \int_{N_{y}}  \Big(\frac{1}{2|x|}\Big)^{n+nq_1/q_2 } dx  \right)^{1/q_1}\\
    &= C_{\vec q,n} s |y|^{-n/q_2}.
    \end{align*}
    Therefore,
    \begin{align*}
    \|F\|_{L^{q_2,\infty}(L^{q_1})} &= \Big\|  \| F\|_{L_{x}^{q_1}}  \Big\|_{L_{y}^{q_2,\infty}}
     \\
        &\leq \sup_{\alpha>0} \alpha \left|\left\{y:    C_{\vec q,n} s |y|^{-n/q_2}> \alpha  \right\}\right|^{1/q_2} \leq C_{\vec q,n} s.
    \end{align*}

    (2) We consider
    $$F_{N}(x,y)=\left(\frac{1}{|x|+|y|}\right)^{ n/q_1 + n/q_2  } \chi_{ \{(x,y):\, 1\leq |y| \leq N \}}, \quad  N\in \mathbb{N}.$$
    Clearly,
    $$\sup\limits_{x,y \in \mathbb{R}^{n}} F_N(x,y) (|x+y|+|x-y|)^{ n/q_1 + n/q_2  } \leq 2^{ n/q_1 + n/q_2  }.$$
    On the other hand,
    \begin{align*}
    \sup_{\alpha>0}  \alpha | \{x:\, |F_{N}(x,y)|>\alpha\}|^{1/q_1}
    &=   \sup_{\alpha>0}  \alpha \left| \left\{x\!:   (|x|+|y|)^{ n/q_1 + n/q_2  }<  \frac{1}{\alpha}  \right \} \right|^{1/q_1} \\
    &\hskip -7mm=  \sup_{\alpha>0}  \alpha  \left| \left\{x\!:  |x|<  (\frac{1}{\alpha})^{q_1q_2/(nq_1+nq_2)} -|y|  \right \} \right|^{1/q_1}.
    \end{align*}
    Choosing $\alpha=(\frac{1}{2|y|})^{ n/q_1 + n/q_2  }$, we have
    \begin{align*}
    \sup_{\alpha>0}  \alpha | \{x:\, |F_N(x,y)|>\alpha\}|^{1/q_1}
    &\geq \Big(\frac{1}{2|y|}\Big)^{ n/q_1 + n/q_2  }  |  \{x:\,  |x|< |y|  \} |^{1/q_1} \\
    &= C_{\vec q,n} |y|^{-n/q_2}.
    \end{align*}
    Hence
    $$ \|F_N\|_{L^{q_2}(L^{q_1,\infty})} =   \Big\|  \| F_N\|_{L_{x}^{q_1,\infty}}  \Big\|_{L_{y}^{q_2}}
    \gtrsim_{\vec q, n} (\int_{1\leq |y| \leq N} \frac{dy}{|y|^{n}})^{1/q_2}=(\ln  N)^{1/q_2}.$$
    Letting $N \to \infty$,   we see a contradiction.

    (3) We use the similar counterexample to show (\ref{eq: endpoint strong weak norm}) and (\ref{eq: endpoint weak strong norm}).
    Let $$F_{N}(x,y)=\left(\frac{1}{|x|+|y|}\right)^{ n/q_1 } \chi^{}_{[1, N] }(|x|)\chi^{}_{[1, N] }(|y|), \quad  N\in \mathbb{N}.$$
    Then $\sup\limits_{x,y \in \mathbb{R}^{n}} F_{N}(x,y) (|x+y|+|x-y|)^{n/q_1} \leq 2^{n/q_1}$.
    While,
    \begin{align*}
    \| F_{N} \|_{L_{x}^{q_1}}
    &=\left(\int_{\{x:\, 1\leq |x| \leq N\}}  \frac{dx}{(|x|+|y|)^{n}}\right)^{1/q_1} \\
    &\geq  \left(\int_{\{x:\, |y| \leq |x| \leq N \}}  \frac{dx}{2^{n}|x|^{n}} \right)^{1/q_1}
    = C_{q_1,n}  (\ln  N - \ln  |y|)^{1/q_1}
    \end{align*}
    and
    $$\| F_{N}\|_{L_{x}^{\infty}} =(1+|y|)^{-n/q_1}.$$
    Consequently,
    $$\|F_{N}\|_{L^{\infty}(L^{q_1})} = \Big\|  \| F_{N}\|_{L_{x}^{q_1}}  \Big\|_{L_{y}^{\infty}} \gtrsim_{q_1,n} (\ln  N)^{1/q_1} $$
    and
    \begin{align*}
    \|F_N\|_{L^{q_1}(L^{\infty})}
    &= \Big\|  \| F_N\|_{L_{x}^{\infty}}  \Big\|_{L_{y}^{q_1}} \\
    &=\left(\int_{\{y:\, 1\leq |y| \leq N\}}  \frac{1}{(1+|y|)^{n}} dy \right)^{1/q_1} \\
    &\geq \left(\int_{\{y:\, 1\leq |y| \leq N\}}  \frac{1}{(2 |y|)^{n}} dy \right)^{1/q_1}
    =C_{q_1,n} (\ln  N)^{1/q_1}.
    \end{align*}
    Letting $N \to \infty$ we have a contradiction.
    That gives the endpoint cases  (\ref{eq: endpoint strong weak norm}) and (\ref{eq: endpoint weak strong norm}).
    Obviously, the other two endpoint cases (\ref{eq: endpoint weak weak norm1}) and (\ref{eq: endpoint weak weak norm2})
    follow from the endpoint cases of (\ref{eq:weak weak norm}).

    \end{proof}

    From (\ref{eq:weak weak norm}) we get the following geometric inequalities  which were studied in \cite{Chen2017}.

    \begin{Corollary} \label{Co:Chen}
    For all $0<p_1, p_2 \leq \infty$,
    \begin{equation}\label{eq: separate weak}
    \|f\|_{L^{p_1,\infty}} \|g\|_{L^{p_2,\infty}}  \leq C_{\vec p,n}  \sup\limits_{x,y \in \mathbb{R}^{n}} f(x) g(y) |x-y|^{ n/p_1 + n/p_2  }
    \end{equation}
    holds for any $f\in L^{p_1, \infty}$, $g\in L^{p_2, \infty}$.

    \medskip

    Furthermore, by interpolation
    \begin{equation}\label{eq: separate}
    \|f\|_{L^{p_1}} \|g\|_{L^{p_2}}  \leq C_{\vec p,n}  \sup\limits_{x,y \in \mathbb{R}^{n}} f(x) g(y) |x-y|^{ n/p_1 + n/p_2  }
    \end{equation}
    holds for any $f\in L^{p_1}$, $g\in L^{p_2}$.
    \end{Corollary}

    \begin{proof}
    Let $F(x,y)=f(x)g(y)$. Obviously
    $$ \|F\|_{L^{p_2,\infty}(L^{p_1,\infty})}  = \|f\|_{L^{p_1,\infty}} \|g\|_{L^{p_2,\infty}}.$$
    By (\ref{eq:weak weak norm}), we obtain that for all $f\in L^{p_1, \infty}$ and $g\in L^{p_2, \infty}$,
    $$\|f\|_{L^{p_1,\infty}} \|g\|_{L^{p_2,\infty}}  \leq C_{\vec p,n}  \sup\limits_{x,y \in \mathbb{R}^{n}} f(x) g(y) (|x+y|+|x-y|)^{ n/p_1 + n/p_2  }.$$

    By interpolation, it is known that
    for all $0< p_1 < r < p_2 \leq \infty$,
    \begin{equation}\label{eq: norm interpolation}
    \|f\|_{r} \leq C_{\vec p,r} \|f\|_{p_1, \infty}^{\theta} \|f\|_{p_2, \infty}^{1-\theta},
    \end{equation}
    where
    $$\frac{1}{r}=\frac{\theta}{p_1}+\frac{1-\theta}{p_2}.$$
    Denote $\gamma= n/p_1 + n/p_2  $.
    It follows from (\ref{eq:weak weak norm}) that
    $$\|f\|_{L^{{n}/{\gamma},\infty}} \|g\|_{L^{\infty}}
    \leq C_{\vec p,n}  \sup\limits_{x,y \in \mathbb{R}^{n}} f(x) g(y) (|x+y|+|x-y|)^{\gamma} $$
    and
    $$\|f\|_{L^{\infty}}  \|g\|_{L^{{n}/{\gamma}, \infty}}
    \leq C_{\vec p,n}  \sup\limits_{x,y \in \mathbb{R}^{n}} f(x) g(y) (|x+y|+|x-y|)|^{\gamma}.$$
    Since
    \[
      \frac{1}{p_1}  = \frac{\gamma}{n} \cdot \frac{p_2}{p_1+p_2} + \frac{1}{\infty}\cdot \frac{p_1}{p_1  + p_2}
      \mathrm{\,\,\,\, and \,\,\,\, }
      \frac{1}{p_2}   = \frac{\gamma}{n} \cdot \frac{p_1}{p_1+p_2} + \frac{1}{\infty}\cdot \frac{p_2}{p_1 + p_2},
    \]
    we see from  (\ref{eq: norm interpolation}) that
    \begin{align*}
    \|f\|_{L^{p_1}} \|g\|_{L^{p_2}}
    &\le C_{\vec p,n} \|f\|_{L^{{n}/{\gamma},\infty}}^{p_2/(p_1+p_2)}
    \|f\|_{L^{\infty}}^{p_1/(p_1+p_2)}
    \|g\|_{L^{{n}/{\gamma},\infty}}^{p_1/(p_1+p_2)}
    \|g\|_{L^{\infty}}^{p_2/(p_1+p_2)}
     \\
    &\le C_{\vec p,n}  \sup\limits_{x,y \in \mathbb{R}^{n}} f(x) g(y) (|x+y|+|x-y|)^{ n/p_1 + n/p_2  }.
    \end{align*}

   From the rearrangement inequality   \cite[Theorem 4.1]{Chen2017}  we have
    \begin{align*}
    & \ \ \ \sup\limits_{x,y \in \mathbb{R}^{n}} f(x) g(y) |x-y|^{ n/p_1 + n/p_2  } \\
    &\geq  \sup\limits_{x,y \in \mathbb{R}^{n}} f^{\ast}(x) g^{\ast}(y) |x-y|^{ n/p_1 + n/p_2  } \\
    &=\frac{1}{2} \sup\limits_{x,y \in \mathbb{R}^{n}} f^{\ast}(x) g^{\ast}(y) |x-y|^{ n/p_1 + n/p_2  }
    +\frac{1}{2} \sup\limits_{x,y \in \mathbb{R}^{n}} f^{\ast}(x) g^{\ast}(y) |x+y|^{ n/p_1 + n/p_2  } \\
    &\geq  \frac{1}{2} \sup\limits_{x,y \in \mathbb{R}^{n}}  f^{\ast}(x) g^{\ast}(y) \big(|x+y|^{ n/p_1 + n/p_2} + |x-y|^{ n/p_1 + n/p_2}) \\
    &\geq   C_{\vec p,n} \sup\limits_{x,y \in \mathbb{R}^{n}} f^{\ast}(x) g^{\ast}(y)
         (|x+y|+|x-y|)^{ n/p_1 + n/p_2  },
    \end{align*}
    where $f^{\ast}$ is the symmetric decreasing rearrangement of $f$.
    Therefore, from
    $$\|f^{\ast}\|_{L^{p_1}} \|g^{\ast}\|_{L^{p_2}} \le C_{\vec p,n}  \sup\limits_{x,y \in \mathbb{R}^{n}} f^{\ast}(x) g^{\ast}(y) (|x+y|+|x-y|)^{ n/p_1 + n/p_2  },$$
    together with the fact that
    $$\|f\|_{L^{p_1}} =\|f^{*}\|_{L^{p_1}},  \|g\|_{L^{p_2}} =\|g^{*}\|_{L^{p_2}},$$
    we get
    \begin{align*}
      \|f\|_{L^{p_1}} \|g\|_{L^{p_2}}
    &=\|f^{*}\|_{L^{p_1}} \|g^{*}\|_{L^{p_2}} \\
    &\leq  C_{\vec p,n}  \sup\limits_{x,y \in \mathbb{R}^{n}} f^{\ast}(x) g^{\ast}(y) (|x+y|+|x-y|)^{ n/p_1 + n/p_2  }  \\
    &\leq C_{\vec p,n} \sup\limits_{x,y \in \mathbb{R}^{n}} f(x) g(y) (|x-y|)^{ n/p_1 + n/p_2  }.
    \end{align*}
   That completes the proof.
    \end{proof}

Next we study the boundedness of $T_{\gamma}$ for general indices.

\begin{Theorem} \label{thm:weak weak norm fractional}
Let $f$ be a nonnegative measurable function defined on $\mathbb{R}^{2n}$.

\begin{enumerate}
\item
 For all $0<q_1 \leq p_1\leq\infty$ and $0<q_2 \leq p_2\leq \infty$ satisfying the homogeneity condition
\begin{equation}\label{eq:homogeneity a}
\frac{1}{q_1}+\frac{1}{q_2}=  \frac{1}{p_1}+\frac{1}{p_2}  +\frac{\gamma}{n},
\end{equation}
we have
\begin{equation}\label{eq:fractional}
\|T_{\gamma} f\|_{L^{q_2,\infty}(L^{q_1,\infty})}  \leq C_{\vec p,\vec q, n} \|f\|_{L^{p_2,\infty}(L^{p_1,\infty})}.
\end{equation}

\item
For all $0<p_1 \leq q_1\leq\infty$ and $0<p_2 \leq q_2\leq\infty$ satisfying the homogeneity condition
\begin{equation}\label{eq:homogeneity b}
 \frac{1}{p_1}+\frac{1}{p_2} = \frac{1}{q_1}+\frac{1}{q_2} +\frac{\gamma}{n},
\end{equation}
we have
\begin{equation}\label{eq:reverse fractional}
\|T^{-1}_{\gamma} f\|_{L^{q_2,\infty}(L^{q_1,\infty})}  \geq C_{\vec p, \vec q,n} \|f\|_{L^{p_2,\infty}(L^{p_1,\infty})}.
\end{equation}

\item
For all $0<q_1 \le p_1 \le  \infty$ and $0<q_2 \le p_2 \le  \infty$ satisfying
$p_1 q_2=p_2 q_1$ and  the homogeneity condition (\ref{eq:homogeneity a}),
we have
\begin{equation}\label{eq:fractional 1}
\|T_{\gamma} f\|_{L^{\vec q, \infty}}  \leq C_{\vec p,\vec q, n} \|f\|_{L^{\vec p,\infty}}.
\end{equation}

\item
For all $0<p_1 \le q_1\le  \infty$ and $0<p_2 \le q_2 \le  \infty$ satisfying
$p_1 q_2=p_2 q_1$ and the homogeneity condition (\ref{eq:homogeneity b}),
we have
\begin{equation}\label{eq:reverse fractional 1}
\|T^{-1}_{\gamma} f\|_{L^{\vec q, \infty}}  \leq C_{\vec p,\vec q, n} \|f\|_{L^{\vec p,\infty}}.
\end{equation}

\end{enumerate}
\end{Theorem}

  \begin{proof}
  (i).  First, we prove (\ref{eq:fractional})  and  (\ref{eq:reverse fractional}).
  Let $g(x,y)= (|x+y|+|x-y|)^{-\gamma} $.
    To prove (\ref{eq:fractional}),
    it suffices to show that  $g\in L^{r_2,\infty}(L^{r_1,\infty})$ for all $0<r_1, r_2 \leq \infty$ with
    $1/r_1 + 1/r_2 = \gamma/n$.

    In fact, if this is the case, set
    \[
     \frac{1}{r_1} :=  \frac{1}{q_1}-\frac{1}{p_1} , \qquad
     \frac{1}{r_2} :=  \frac{1}{q_2}-\frac{1}{p_2}  .
    \]
    Then we have $1/r_1 + 1/r_2 = \gamma/n$.
    Applying  H\"older's inequality gives   that
    \begin{equation}\label{eq:holder inequality}
    \|T_{\gamma} f\|_{L^{q_2,\infty}(L^{q_1,\infty})} \leq  \|f\|_{L^{p_2,\infty}(L^{p_1,\infty})}  \|g\|_{L^{r_2,\infty}(L^{r_1,\infty})},
    \end{equation}
    which is (\ref{eq:fractional}).
    Note if $p_1=\infty$ or $p_2=\infty$, (\ref{eq:holder inequality}) follows easily.

    First, we consider the endpoints.
    When $r_1=\infty$ and $r_2=n/\gamma$, we have
    $$(|x+y|+ |x-y|)^{-\gamma} \leq \frac{1}{(2|y|)^{\gamma}},  \ \forall \ y \neq 0.$$
    So
    $$\|g\|_{L^{(n/\gamma, \infty)}(L^{\infty})}  \leq C_{n, \gamma}.$$
    When $r_1=n/\gamma$ and  $r_2=\infty$, we have
    \[
      \sup_{\alpha>0}   \alpha | \{x:\, (|x+y|+|x-y|)^{-\gamma}>\alpha \} |^{\gamma /n} \leq C_{n, \gamma},
    \quad \forall \ y\in \mathbb{R}^{n}.
    \]

    It remains to verify the case when $0<r_1, r_2 <\infty$.
    If $|y|< \alpha^{-1/ \gamma}$,
    $$\alpha | \{x:\, |g(x,y)|>\alpha \} |^{1/r_1} \leq C_{r_1,n} \alpha^{1-n/(\gamma r_1)}.$$
    If $|y| \geq \alpha^{-1/ \gamma}$,  it is not hard to see that $\alpha | \{x:\, |g(x,y)|>\alpha \} |^{1/r_1} =0$.
    Thus for any $y \neq 0$,
    $$\sup_{\alpha>0}  \alpha | \{x:\, |g(x,y)|>\alpha \}|^{1/r_1} \leq C_{r_1,n} |y|^{n/r_1-\gamma}.$$
    Then by definition,
    \begin{align*}
    \|g\|_{L^{r_2,\infty}(L^{r_1,\infty})}
    &= \Big\|  \| g\|_{L_{x}^{r_1,\infty}}  \Big\|_{L_{y}^{r_2,\infty}} \\
    &=\sup_{\beta>0} \beta \left|\left\{y:\, \sup_{\alpha>0}  \alpha | \{x:\, |g(x,y)|>\alpha\}|^{1/r_1} > \beta  \right\}\right|^{1/r_2} \\
    &\leq  \sup_{\beta>0} \beta \left|\left\{y:\,  C_{r_1,n}  |y|^{n/r_1-\gamma} > \beta  \right\}\right|^{1/r_2}
    =C_{\vec r,n}.
    \end{align*}
    Hence (\ref{eq:fractional}) is true.

    On the other hand,   we see from (\ref{eq:fractional})  that
    \begin{align*}
    \|f\|_{L^{p_2,\infty}(L^{p_1,\infty})}
    &=\| T_{\gamma} T^{-1}_{\gamma} f\|_{L^{p_2,\infty}(L^{p_1,\infty})}
     \lesssim \|T^{-1}_{\gamma} f\|_{L^{q_2,\infty}(L^{q_1,\infty})}.
    \end{align*}
    Hence  (\ref{eq:reverse fractional}) is also true.

    (ii). Next we prove  (\ref{eq:fractional 1})  and   (\ref{eq:reverse fractional 1}).
    Suppose that $0<p_1,p_2, q_1, q_2 < \infty$ and $1/q_1+1/q_2=1/p_1+1/p_2+\gamma/n$.
    Then there always exist $r_1$, $r_2$ satisfying  $1/r_1+1/r_2=\gamma/n$ such that
    $p_1/p_2=r_1/r_2$.

    In retrospect, it follows from
    the proof of Theorem~\ref{thm:interpolation} that
    for all $\vec r=(r_1, r_2)$ satisfying $1/r_1+1/r_2=\gamma/n$,
    $$(|x+y|+|x-y|)^{-\gamma} \in L^{\vec r, \infty}.$$
    By H\"older's inequality in Theorem~\ref{thm: holder mixed weak norm},
    $$\|T_{\gamma} f\|_{L^{\vec q,\infty}}  \lesssim_{\vec p, \vec q, n}   \|f\|_{L^{\vec p,\infty}}  \|(|x+y|+|x-y|)^{-\gamma}\|_{L^{\vec r, \infty}},$$
    which gives (\ref{eq:fractional 1}).
    When $p_1=q_1=\infty$, $p_2=q_2=\infty$, or $p_1=p_2=\infty$,
    as shown in Theorem~\ref{thm: holder mixed weak norm}, H\"older's inequality still holds for these  cases.
    Taking similar arguments we conclude (\ref{eq:fractional 1}).

    Likewise, it follows from (\ref{eq:fractional 1}) that
    \begin{align*}
    \|f\|_{L^{\vec p,\infty}}
    &=\| T_{\gamma} T^{-1}_{\gamma} f\|_{L^{\vec p,\infty}}    \lesssim \|T^{-1}_{\gamma} f\|_{L^{\vec q,\infty}}.
    \end{align*}
    Hence  (\ref{eq:reverse fractional 1}) is also true.

    \end{proof}

    \begin{Remark}
    We point out that the condition
    $1/q_1 + 1/q_2= 1/p_1 + 1/p_2 +  \gamma/n$ follows from the homogeneity.
    \end{Remark}

    Suppose that (\ref{eq:fractional}) is true.
    Then we consider functions $f(\frac{\cdot}{R}, \frac{\cdot}{R})$ for   $R>0$.
    On the one hand,
    $$\|f(\frac{\cdot}{R}, \frac{\cdot}{R})\|_{L^{p_2,\infty}(L^{p_1,\infty})}= R^{ n/p_1 + n/p_2  } \|f\|_{L^{p_2,\infty}(L^{p_1,\infty})}.$$
    On the other hand,
     \begin{align*}
    &\ \ \ \sup_{\alpha>0}  \alpha \left| \left\{x:\, f(\frac{x}{R},\frac{y}{R}) (|x+y|+ |x-y|)^{-\gamma} >\alpha \right\} \right|^{1/q_1} \\
    &=\sup_{\alpha>0}  \alpha  \left| \left\{x:\, f(\frac{x}{R},\frac{y}{R}) \left( \left|\frac{x}{R}+\frac{y}{R}\right|+ \left|\frac{x}{R}-\frac{y}{R} \right| \right)^{-\gamma} >R^{\gamma}\alpha \right\} \right|^{1/q_1} \\
    &=\sup_{\alpha>0}  \alpha R^{n/q_1} \left| \left\{x:\, f(x,\frac{y}{R}) \left(\left|x+\frac{y}{R}\right|+ \left|x-\frac{y}{R}\right|\right)^{-\gamma} >R^{\gamma}\alpha \right\} \right|^{1/q_1} \\
    &=R^{ n/q_1 -\gamma} \sup_{\alpha>0}  \alpha  \left| \left\{x:\, f(x,\frac{y}{R}) \left(\left|x+\frac{y}{R}\right|+ \left|x-\frac{y}{R}\right|\right)^{-\gamma} >\alpha \right\} \right|^{1/q_1}\\
    & = R^{ n/q_1 -\gamma}  \|T_{\gamma} f(\cdot , \frac{y}{R}) \|_{L^{q_1,\infty}_x}.
    \end{align*}
    Hence
    \begin{align*}
    &\hskip -10mm \Big\|T_{\gamma} f(\frac{\cdot}{R}, \frac{\cdot}{R})\Big\|_{L^{q_2,\infty}(L^{q_1,\infty})} \\
    &=R^{n/q_1-\gamma}  \sup\limits_{\beta>0}  \beta \left| \left\{y:\,
      \|T_{\gamma} f(\cdot, \frac{y}{R}) \|_{L^{q_1,\infty}_x} >\beta \right\} \right|^{1/q_2} \\
    &=R^{n/q_1 + n/q_2  -\gamma}   \sup\limits_{\beta>0}  \beta | \{y:\,
     \|T_{\gamma} f(\cdot , y) \|_{L^{q_1,\infty}_x}
     >\beta\}|^{1/q_2} \\
    &=R^{n/q_1 + n/q_2 -\gamma} \|T_{\gamma} f\|_{L^{q_2,\infty}(L^{q_1,\infty})}.
    \end{align*}
    These imply that for all $R>0$,
    $$R^{n/q_1 + n/q_2 -\gamma} \lesssim R^{ n/p_1 + n/p_2  },$$
    which gives the condition
    $$\frac{1}{q_1}+\frac{1}{q_2}=\frac{1}{p_1}+\frac{1}{p_2}  +\frac{\gamma}{n}.$$

\begin{Remark}
We illustrate that
(\ref{eq:fractional 1}) might be false if $p_1q_2  \ne p_2q_1$.
\end{Remark}

\begin{proof}
(i).\,\,  $p_1=q_1$ and $p_2 > q_2$.

 Let $f(x,y)= (|x+y|+|x-y|)^{\gamma} \chi^{}_{E}(x,y)$,
  where $E=\{(x,y):\, 0<|x|<|y|^{-\alpha}, 1\le |y|\le N\}$, $\alpha$ and $N$ are constants.
Then we have
\begin{align*}
\|T_{\gamma}f\|_{L^{\vec q, \infty}}
&=  \|\chi^{}_{E} \|_{L^{\vec q}}
   = \Big (\int_{1\le |y|\le N}  |y|^{-n \alpha q_2/q_1} dy\Big)^{1/q_2} \\
&\approx  \Big (\int_1^N  t^{-n \alpha q_2/q_1+n-1} dt\Big)^{1/q_2}
\end{align*}
and
\begin{align*}
\|f\|_{L^{\vec p}}
  &\le \Big( \int_{1\le |y|\le N}  |y|^{\gamma p_2 -n \alpha p_2/p_1} dy\Big)^{1/p_2}
  \lesssim \Big( \int_1^N   t^{\gamma p_2 -n \alpha p_2/p_1+n-1} dt\Big)^{1/p_2}.
\end{align*}
Set $\alpha = p_1(\gamma/n+1/ p_2)$. We have
\[
 \frac \alpha {p_1} = \frac{\gamma}{n} + \frac{1}{p_2} = \frac{1}{q_1} + \frac{1}{q_2} - \frac{1}{p_1} = \frac{1}{q_2}.
\]
Therefore, $\alpha q_2/q_1 = 1$. It follows that
\[
\|T_{\gamma}f\|_{L^{\vec q, \infty}}\approx (\ln N)^{1/q_2}
\mathrm{\,\,\,\, and\,\,\,\,}
  \|f\|_{L^{\vec p}}\approx (\ln N)^{1/p_2}.
\]
Hence
\[
    \lim_{N\rightarrow \infty} \frac{\|T_{\gamma}f\|_{L^{\vec q, \infty}}}{\|f\|_{L^{\vec p}}} = \infty.
\]

(ii).\,\, $p_2=q_2$ and $p_1>q_1$.

In this case, we have
$$\|T_{\gamma}f\|_{L^{\vec q, \infty}} \lesssim \|f\|_{L^{\vec p}}.$$

In fact,   there is some constant $C$ such that
\[
  \left\|(|x+y|+|x-y|)^{-\gamma}\right\|_{L_{x}^{n/\gamma, \infty}} \le C,\qquad \forall y.
\]
Now we see from H\"older's inequality that
$$\|T_{\gamma}f\|_{L_{x}^{q_1, \infty}} =\|f(x,y)(|x+y|+|x-y|)^{-\gamma} \|_{L_{x}^{q_1, \infty}} \lesssim \|f\|_{L_{x}^{p_1, \infty}}.$$
Hence
$$\|T_{\gamma}f\|_{L_{y}^{q_2} (L_{x}^{q_1, \infty} )}
\leq  \|f\|_{L_{y}^{p_2}(L_{x}^{p_1, \infty})}.$$
Therefore,
$$\|T_{\gamma}f\|_{L^{\vec q, \infty}} \lesssim \|f\|_{L^{\vec p}}.$$

However, (\ref{eq:fractional 1}) is  not true.
To see this, let
$$f(x,y)= (|x+y|+|x-y|)^{\gamma} \chi^{}_{E}(x,y),$$
where
$$E=\{(x,y):\, 0<|y|<1, |x| \leq |y|^{-\beta}-|y|\}, \qquad  \beta= \frac{q_1}{q_2}.$$
Then
\begin{align*}
\|T_{\gamma}f\|_{L^{\vec q, \infty}}
&=  \|\chi^{}_{E} \|_{L^{\vec r}} \\
&\gtrsim \Big(\int_{0<|y|<1/2}  \big( |y|^{-\beta}-1 \big)^{n q_2/q_1} dy\Big)^{1/q_2}  \\
&\gtrsim \Big(\int_{0<|y|<1/2}   |y|^{-\beta n q_2/q_1}   dy\Big)^{1/q_2}  \\
&=  \infty
\end{align*}
and
\begin{align*}
\|f\|_{L^{\vec p, \infty}}
&=\sup\limits_{\lambda>0} \lambda \int_{ |y|<1} \left(\left|\{x:\,  (|x+y|+|x-y|)^{\gamma} \chi^{}_{E}(x,y) >\lambda\}\right|^{p_2/p_1} dy\right)^{1/p_2} \\
&\le \sup\limits_{\lambda>0} \lambda \left(\int_{ |y|<1} \left|\left\{x:\, (|x| +|y|) \chi^{}_{E}(x,y) >2 ^{-1} \lambda^{1/\gamma} \right\}\right|^{p_2/p_1} dy\right)^{1/p_2} \\
 &\le \sup\limits_{\lambda>0} \lambda \left(\int_{|y|^{-\beta}>2^{-1}\lambda^{1/\gamma}} |y|^{-\beta n p_2/p_1} dy\right)^{1/p_2} \\
&\le C_{\vec p, n, \gamma} \sup\limits_{\lambda>0} \lambda \lambda^{n/(p_1 \gamma)-n/(\beta p_2 \gamma)}.
\end{align*}
Note
$$\beta=\frac{q_1}{q_2}= \frac{np_1}{np_2+p_1p_2 \gamma}.$$
So
$$\frac{n}{p_1 \gamma}-\frac{n}{\beta p_2 \gamma}=-1.$$
Thus we have $f\in L^{\vec p, \infty}(\bbR^n \times \bbR^n)$.
Therefore,  (\ref{eq:fractional 1}) is false is this case.
 \end{proof}

\medskip

Next we study the boundedness of $L_{\gamma}$.
    Observe that
    \[
       \sup\limits_{\alpha>0} \alpha | \{x:\,
            |x-y|^{-\gamma}>\alpha \} |^{\gamma/n} \leq C_{n,\gamma},\qquad \forall y.
    \]
    we have
    $$|x-y|^{-\gamma} \in L^{\infty}(L^{n/\gamma,\infty}).$$
    On the other hand, it is easy to see that
    $$|x-y|^{-\gamma} \notin L^{n/\gamma,\infty}(L^{\infty}).$$
    Therefore, by H\"older's inequality we have the following inequalities.

\begin{Theorem} \label{thm:weak strong norm fractional}
Let $f$ be a nonnegative measurable function  defined on $\mathbb{R}^{2n}$.
Then for all $0<r< p_1\leq \infty$ and $0<p_2\leq \infty$  satisfying the homogeneity condition $1/r =1/p_1 +\gamma/n$,
\begin{equation}\label{eq: half fractional}
\|L_{\gamma} f\|_{L^{p_2}(L^{r,\infty})}  \leq C_{\vec p,r,n} \|f\|_{L^{p_2}(L^{p_1,\infty})},
\end{equation}
\begin{equation}\label{eq: half fractional 1}
\|L_{\gamma} f\|_{L^{p_2, \infty}(L^{r,\infty})}  \leq C_{\vec p,r,n} \|f\|_{L^{p_2, \infty}(L^{p_1,\infty})}.
\end{equation}
And for all $0< p_1< r \leq \infty$ and $0<p_2\leq \infty$   satisfying the homogeneity condition $1/p_1 =1/r+\gamma/n $,
\begin{equation} \label{eq: half reverse fractional}
\| L_{\gamma}^{-1} f\|_{L^{p_2}(L^{r,\infty})}  \geq C_{\vec p,r,n} \|f\|_{L^{p_2}(L^{p_1,\infty})},
\end{equation}
\begin{equation} \label{eq: half reverse fractional 1}
\| L_{\gamma}^{-1} f\|_{L^{p_2, \infty}(L^{r,\infty})}  \geq C_{\vec p,r,n} \|f\|_{L^{p_2, \infty}(L^{p_1,\infty})}.
\end{equation}
However, for any multiple indices $\vec p$ and $\vec q$,
\begin{align}
  \|L_{\gamma} f\|_{L^{q_2,\infty}(L^{q_1})}   &  \not\leq C_{\vec p,\vec q,n} \|f\|_{L^{p_2,\infty}(L^{p_1})};
     \label{eq: false half fractional} \\
\|L_{\gamma} f\|_{L^{q_2}(L^{q_1,\infty})}   &  \not\leq C_{\vec p, \vec q,n} \|f\|_{L^{p_2}(L^{p_1, \infty})} \  unless  \, \, p_2=q_2;
     \label{eq: false half fractional b} \\
\|L_{\gamma} f\|_{L^{q_2, \infty}(L^{q_1,\infty})}   &  \not\leq C_{\vec p,\vec q,n} \|f\|_{L^{p_2, \infty}(L^{p_1, \infty})} \  unless  \, \, p_2=q_2;
 \label{eq: false half fractional c} \\
\|L_{\gamma} f\|_{L^{\vec q}}
  &\not\leq C_{\vec p,\vec q,n} \|f\|_{L^{\vec p}},  \label{eq: false half fractional d}
\end{align}

\end{Theorem}

\begin{proof}
(i). First, we prove (\ref{eq: half fractional})-(\ref{eq: half reverse fractional 1}).
As before, the condition $1/r =1/p_1 +\gamma/n$ follows from the homogeneity.
Let $h(x,y)=|x-y|^{-\gamma}$.
    Observe that
    \[
       \sup\limits_{\alpha>0} \alpha | \{x:\,
            |x-y|^{-\gamma}>\alpha \} |^{\gamma/n} \leq C_{n,\gamma},\qquad \forall y.
    \]
we have
$h \in L^{\infty}(L^{n/\gamma,\infty})$.

For $0<r<p_1<\infty$ with $1/r=1/p_1+\gamma/n$, we see from H\"older's inequality that
\begin{equation}\label{eq:separate holder}
\| L_{\gamma} f(\cdot,y)\|_{L_{x}^{r,\infty}} \lesssim_{p_1,r,n}  \|f(\cdot,y)\|_{L_{x}^{p_1,\infty}}  \|h(\cdot,y)\|_{L_{x}^{n/\gamma, \infty}}.
\end{equation}
And for $p_1=\infty$,
$\| L_{\gamma} f(\cdot,y)\|_{L_{x}^{n/\gamma,\infty}} \leq \|f(\cdot,y)\|_{L_{x}^{\infty}}
   \|h(\cdot,y)\|_{L_{x}^{n/\gamma, \infty}}$.
It follows that for all $0<p_2\leq \infty$,
\begin{align*}
\Big\|  \| L_{\gamma} f\|_{L_{x}^{r,\infty}}  \Big\|_{L_{y}^{p_2}}
&= \Big\|    \| L_{\gamma} f\|_{L_{x}^{r,\infty}}  \Big\|_{L_{y}^{p_2}} \\
&\lesssim_{p_1, r, n} \Big\|    \|f\|_{L_{x}^{p_1,\infty}}  \|h\|_{L_{x}^{n/\gamma, \infty}}    \Big\|_{L_{y}^{p_2}} \\
&\leq \Big\|    \|f\|_{L_{x}^{p_1,\infty}}  \Big\|_{L_{y}^{p_2}}  \Big\|    \|h\|_{L_{x}^{n/\gamma,\infty}}  \Big\|_{L_{y}^{\infty}}.
\end{align*}
That is,
$$\|L_{\gamma} f\|_{L^{p_2}(L^{r,\infty})}  \leq C_{\vec p, r, n}
  \|f\|_{L^{p_2}(L^{p_1,\infty})} \|h\|_{L^{\infty}(L^{n/\gamma,\infty})}.$$
Meanwhile, by H\"older's inequality we also obtain
$$\|L_{\gamma} f\|_{L^{p_2, \infty}(L^{r,\infty})}  \leq C_{\vec p, r, n}
 \|f\|_{L^{p_2, \infty}(L^{p_1,\infty})} \|h\|_{L^{\infty}(L^{n/\gamma,\infty})}.$$

Furthermore, applying the similar arguments as  in the proof of (\ref{eq:reverse fractional}) we get
their reverse versions (\ref{eq: half reverse fractional}) and (\ref{eq: half reverse fractional 1}).

(ii).\,\, Next we prove (\ref{eq: false half fractional}).

First, we assume that $p_2>q_2$.
Let $M>0$ and consider
$$f(x,y)=|x-y|^{\gamma-n/q_1} \chi_{\{(x,y):\, |x|, |y| \leq M\}}.$$
Below we calculate
$$\|f\|_{L_{x}^{p_1}}= \Big(\int_{\{  |x| \leq M\}} |x-y|^{(\gamma -n/q_1)p_1}dx \Big)^{1/p_1}.$$
For $\gamma  >n/q_1$, we have
$$\Big(\int_{\{  |x| \leq M\}} |x-y|^{(\gamma -n/q_1)p_1} dx \Big)^{1/p_1} \leq C_{n, \gamma, \vec p} M^{\gamma}.$$
And for $0< \gamma \le n/q_1$, since $p_2>q_2$, we have $n/q_1 - \gamma < n/p_1$. It follows that
\begin{align*}
&\hskip -10mm\Big(\int_{\{ |x| \leq M\}}   |x-y|^{(\gamma -n/q_1)p_1} dx\Big)^{1/p_1} \\
&\leq \Big( \sum_{2^{j}\leq M} \int_{\{x:\, 2^{j}<|x-y|\leq 2^{j+1} \}}
     |x-y|^{(\gamma -n/q_1)p_1} dx\Big)^{1/p_1} \\
&\leq C_{n, \gamma, p_2} \Big(\sum_{2^{j}\leq M}  \frac{2^{jn}}{2^{j (n/q_1 - \gamma )p_1  }} \Big)^{1/p_1} <\infty.
\end{align*}
Hence $\|f\|_{L^{p_1,\infty}(L^{p_2})}<\infty$.

On the other hand, observe that
$$\|L_{\gamma} f\|_{L_{x}^{q_1}} = \Big(\int_{\{  |x| \leq M\}} \frac{1}{|x-y|^{n}} dx \Big)^{1/q_1}
 =\infty.$$
We have  $\|L_{\gamma} f\|_{L^{q_2,\infty}(L^{q_1})} = \infty$.
This  proves  (\ref{eq: false half fractional}) for $p_2>q_2$.

For the case of $p_2 \leq q_2$, consider
\[
  f(x,y) = |x-y|^{\gamma-n/q_1} \Big|\ln |x-y| \Big|^{-1/q_1}  \chi_{\{(x,y):\, |x|, |y| \leq 1/3\}}.
\]
Observe that $p_1>q_1$. We have
\[
  \|L_{\gamma} f\|_{L_{x}^{q_1}} = \Big(\int_{\{  |x| \leq 1/3\}} \frac{1}{|x-y|^{n} \left|\ln |x-y|\right|} dx \Big)^{1/q_1}
 =\infty,
\]
while
\[
  \|  f\|_{L_{x}^{p_1}} = \Big(\int_{\{  |x| \leq 1/3\}} \frac{1}{|x-y|^{n} \left|\ln |x-y|\right|^{p_1/q_1}} dx \Big)^{1/p_1} \le C < \infty,\quad |y|<1/3.
\]
Again, we get  $\|L_{\gamma} f\|_{L^{q_2,\infty}(L^{q_1})} = \infty$ and $\| f\|_{L^{p_2,\infty}(L^{p_1})}<\infty$.

(iii).\,\, We use the same counterexample to prove (\ref{eq: false half fractional b}) and (\ref{eq: false half fractional c}).
From  the homogeneity it suffices to consider the case
$$\frac{1}{q_1}+\frac{1}{q_2}=\frac{1}{p_1}+\frac{1}{p_2}+\frac{\gamma}{n}.$$
Let $f(x,y)=|x-y|^{-n/p_{1}} \chi_{[0,1]}(|y|)$, then
\begin{align*}
\|f\|_{L_{x}^{p_1, \infty}}
&=\sup\limits_{\alpha>0} \alpha |\{x: |x-y|^{-n/p_{1}}\chi_{[0,1]}(|y|) >\alpha\}|^{1/p_1} \\
&=\sup\limits_{\alpha>0} \alpha |\{x: |x-y|<(1/\alpha)^{p_1/n}\}|^{1/p_1}\chi_{[0,1]}(|y|) \\
&=C_{p_1, n}\chi_{[0,1]}(|y|) .
\end{align*}
On the other hand, we have
\begin{align*}
 \|L_{\gamma} f\|_{L_{x}^{q_1, \infty}}
&=\sup\limits_{\alpha>0} \alpha |\{x: |x-y|^{-n/p_{1}} |x-y|^{-\gamma}>\alpha\}|^{1/q_1} \chi_{[0,1]}(|y|)\\
&=\sup\limits_{\alpha>0} \alpha (1/\alpha)^{n p_1/(n q_1+\gamma p_1 q_1)}\chi_{[0,1]}(|y|) \\
&=\sup\limits_{\alpha>0} \alpha^{1-n p_1/(n q_1+\gamma p_1 q_1)}=\infty,   \qquad |y|<1,
\end{align*}
unless
$$n p_1/(n q_1+\gamma p_1 q_1)=1, \, \mathrm{i.e.}, \, p_2=q_2.$$

(iv).\,\, Finally, we prove (\ref{eq: false half fractional d}).
First, we see from  the homogeneity that
it suffices to consider the case
$$\frac{1}{q_1}+\frac{1}{q_2}=\frac{1}{p_1}+\frac{1}{p_2}+\frac{\gamma}{n}.$$
There are two cases.

(a) $(\gamma-n/q_{1})p_1>-n$.

In this case, set
$$f(x,y)=|x-y|^{\gamma-n/q_{1}} \chi_{\{(x,y):\, |y| \leq |x| \leq N, 1 \leq |y| \leq 2 \}}, \quad N>10.$$

If $(\gamma-n/q_{1})p_1 \geq 0$, then
\begin{align*}
\|f\|_{L^{\vec p}}
&=\left(\int_{\{y:\, 1 \leq |y| \leq 2\}} \Big( \int_{\{x:\, |y| \leq |x| \leq N\}} |x-y|^{(\gamma-n/q_{1})p_{1}} dx \Big)^{p_2/p_1} dy \right)^{1/p_{2}} \\
&\leq C_{\vec p, \vec q, n} \left(\int_{\{y:\, 1 \leq |y| \leq 2\}} \Big( N^{(\gamma-n/q_{1})p_{1}+n}\Big)^{p_2/p_1} dy  \right)^{1/p_{2}} \\
&\leq C_{\vec p, \vec q, n} N^{n/q_2-n/p_2},
\end{align*}
while
$$\|L_{\gamma} f  \|_{L^{\vec q}}
=\left(\int_{\{  1 \leq |y| \leq 2\}} \Big( \int_{\{x:\, |y| \leq |x| \leq N\}} |x-y|^{-n} dx \Big)^{q_2/q_1} dy \right)^{1/q_{2}}=\infty.$$

If $-n<(\gamma-n/q_{1})p_1<0$, then
\begin{align*}
\|f\|_{L^{\vec p}}
&=\left(\int_{\{  1 \leq |y| \leq 2\}} \Big( \int_{\{x:\, |y| \leq |x| \leq N\}} |x-y|^{(\gamma-n/q_{1})p_{1}} dx\Big)^{p_2/p_1} dy \right)^{1/p_{2}} \\
&\leq \left(\!\int_{\{  1 \leq |y| \leq 2\}} \!\Big(\! \sum\limits_{2^{j}\leq N}
  \int_{\{x:\, 2^j \leq |x-y| \leq 2^{j+1} \}} \!\! |x-y|^{(\gamma-n/q_{1})p_{1}} dx \Big)^{p_2/p_1}\! dy \! \right)^{1/p_{2}}  \\
&\leq C_{\vec p, \vec q, n} \left(\int_{\{  1 \leq |y| \leq 2\}}  \Big( \sum\limits_{2^{j}\leq N} 2^{jn+j(\gamma-n/q_{1})p_1} \Big)^{p_2/p_1} dy \right)^{1/p_{2}}  \\
&< \infty.
\end{align*}
On the other hand,
$$\|L_{\gamma} f  \|_{L^{\vec q}}
=\left(\int_{1 \leq |y| \leq 2}\Big( \int_{|y| \leq |x| \leq N} |x-y|^{-n} dx\Big)^{q_2/q_1} dy \right)^{1/q_{2}}=\infty.$$

(b) $(\gamma-n/q_{1})p_1\le -n$.

In this case, we consider the function
$$f(x,y)=|x-y|^{\gamma-n/q_{1}} \chi_{\{(x,y):\, 2|y| \leq |x| \leq N, 1 \leq |y| \leq 2 \}}, \quad N>10.$$

If $(\gamma-n/q_{1})p_1<-n$,
it is not hard to see
\begin{align*}
\|f\|_{L^{\vec p}}
&=\left(\int_{\{y:\, 1 \leq |y| \leq 2\}} \Big( \int_{\{x:\, 2|y| \leq |x| \leq N\}} |x-y|^{(\gamma-n/q_{1})p_{1}} dx \Big)^{p_2/p_1} dy \right)^{1/p_{2}} \\
&\leq \left(\int_{\{y:\, 1 \leq |y| \leq 2\}} \Big( \int_{\{x:\, 2|y| \leq |x| \leq N\}} (|x|-|y|)^{(\gamma-n/q_{1})p_{1}} dx \Big)^{p_2/p_1} dy \right)^{1/p_{2}} \\
&\lesssim_{\vec p, \vec q, n} \left(\int_{\{y:\, 1 \leq |y| \leq 2\}} \Big( \int_{\{x:\, 2|y| \leq |x| \leq N\}} |x|^{(\gamma-n/q_{1})p_{1}} dx \Big)^{p_2/p_1} dy \right)^{1/p_{2}} \\
&=  \left(\int_{\{y:\, 1 \leq |y| \leq 2\}}  \Big( (2|y|)^{(\gamma-n/q_{1})p_{1}+n}
        -N^{(\gamma-n/q_{1})p_{1}+n} \Big)^{p_2/p_1}  dy \right)^{1/p_{2}} \\
&\le C_{\vec p, \vec q, n}.
\end{align*}
However,
\begin{align*}
\|L_{\gamma} f  \|_{L^{\vec q}}
&=\left(\int_{\{y:\, 1 \leq |y| \leq 2\}}\Big( \int_{\{x:\, 2|y| \leq |x| \leq N\}} |x-y|^{-n} dx\Big)^{q_2/q_1} dy \right)^{1/q_{2}}\\
&\geq \left(\int_{\{y:\, 1 \leq |y| \leq 2\}}\Big( \int_{\{x:\, 2|y| \leq |x| \leq N\}} (|x|+|y|)^{-n} dx\Big)^{q_2/q_1} dy \right)^{1/q_{2}}\\
&\geq C_{\vec p, \vec q, n} \left(\int_{\{y:\, 1 \leq |y| \leq 2\}}\Big( \int_{\{x:\, 2|y| \leq |x| \leq N\}} |x|^{-n} dx\Big)^{q_2/q_1} dy \right)^{1/q_{2}}\\
&=C_{\vec p, \vec q, n} \left(\int_{\{y:\, 1 \leq |y| \leq 2\}} \Big(\mathrm{ln} N-\mathrm{ln} (2|y|) \Big)^{q_2/q_1} dy \right)^{1/q_{2}}\\
&\geq C_{\vec p, \vec q, n} (\mathrm{ln} N)^{1/q_{1}}.
\end{align*}
Letting $N \to \infty$ we get a contradiction.

If $(\gamma-n/q_{1})p_1=-n$, then
$p_2=q_2$ and $p_1> q_1$.
We have already obtained
\begin{align*}
\|L_{\gamma} f  \|_{L^{\vec q}}
&=\left(\int_{\{y:\, 1 \leq |y| \leq 2\}}\Big( \int_{\{x:\, 2|y| \leq |x| \leq N\}} |x-y|^{-n} dx\Big)^{q_2/q_1} dy \right)^{1/q_{2}}\\
&\geq C_{\vec p, \vec q, n} (\mathrm{ln} N)^{1/q_{1}}.
\end{align*}
On the other hand, in this case
\begin{align*}
\|f\|_{L^{\vec p}}
&=\left(\int_{\{y:\, 1 \leq |y| \leq 2\}} \Big( \int_{\{x:\, 2|y| \leq |x| \leq N\}} |x-y|^{(\gamma-n/q_{1})p_{1}} dx \Big)^{p_2/p_1} dy \right)^{1/p_{2}} \\
&=\left(\int_{\{y:\, 1 \leq |y| \leq 2\}} \Big( \int_{\{x:\, 2|y| \leq |x| \leq N\}} |x-y|^{-n} dx \Big)^{p_2/p_1} dy \right)^{1/p_{2}} \\
&\leq \left(\int_{\{y:\, 1 \leq |y| \leq 2\}} \Big( \int_{\{x:\, 2|y| \leq |x| \leq N\}} (|x|-|y|)^{-n} dx \Big)^{p_2/p_1} dy \right)^{1/p_{2}} \\
&\leq C_{\vec p, \vec q, n} \left(\int_{\{y:\, 1 \leq |y| \leq 2\}} \Big( \int_{\{x:\, 2|y| \leq |x| \leq N\}} |x|^{-n} dx \Big)^{p_2/p_1} dy \right)^{1/p_{2}} \\
&= C_{\vec p, \vec q, n} \left(\int_{\{y:\, 1 \leq |y| \leq 2\}}  \Big(\mathrm{ln} N-\mathrm{ln} (2|y|) \Big)^{p_2/p_1} dy \right)^{1/p_{2}}\\
&\leq C_{\vec p, \vec q, n} (\mathrm{ln} N)^{1/p_{1}}.
\end{align*}
Note $p_1> q_1$. By letting   $N \to \infty$,   we get a  contradiction.
\end{proof}

At the end of this paper, we show that  Theorem~\ref{thm:weak weak norm fractional} implies the classical
    Hardy-Littlewood-Sobolev inequality and its reverse version as follows.

\begin{Corollary}\label{cor:H-L-S inequality}
    For $1< p_1, p_2< \infty$ with  $1/p_1 +1/p_2   > 1$,
    \begin{equation}\label{eq: separate fractional}
    \int_{\mathbb{R}^{n}}  \int_{\mathbb{R}^{n}}  f(x) g(y) |x-y|^{-n(2-1/p_1 -1/p_2  )} dxdy \leq C_{\vec p,n}  \|f\|_{L^{p_1}} \|g\|_{L^{p_2}}
    \end{equation}
    holds for all nonnegative functions $f\in L^{p_1}$, $g\in L^{p_2}$.

    For $0< p_1, p_2< 1$ and all nonnegative functions  $f\in L^{p_1}$, $g\in L^{p_2}$,
    \begin{equation}\label{eq: reverse separate fractional}
    \int_{\mathbb{R}^{n}}  \int_{\mathbb{R}^{n}}  f(x) g(y) |x-y|^{n(1/p_1 +1/p_2  -2)} dxdy \geq C_{\vec p,n}  \|f\|_{L^{p_1}} \|g\|_{L^{p_2}}.
    \end{equation}
\end{Corollary}

\begin{proof}
   It is known that the Hardy-Littlewood-Sobolev inequality and the reverse Hardy-Littlewood-Sobolev inequality  are equivalent~\cite{Gressman2011}.
   It suffices to show one of them is true.
   Below show that   Theorem~\ref{thm:weak weak norm fractional} implies   (\ref{eq: reverse separate fractional}).

    Denote $\gamma= n/p_1 + n/p_2-2n$. It follows from Theorem~\ref{thm:weak weak norm fractional} that
    $$\|T^{-1}_{\gamma} (f\otimes g) \|_{L^{q_2,\infty}(L^{q_1,\infty})}  \geq C_{\vec p, \vec q, n} \|(f\otimes g) \|_{L^{p_2,\infty}(L^{p_1,\infty})}$$
    holds for all $0<p_1 \leq q_1\leq\infty$ and $0<p_2 \leq q_2\leq\infty$ satisfying $1/p_1+1/p_2=1/q_1+1/q_2 +\gamma/n$.
    Hence for all $0<p_1, p_2\le 1$ with $1/p_1+1/p_2=2+\gamma/n$,
    $$\|T^{-1}_{\gamma} (f\otimes g) \|_{L^1(L^{1})}\geq  \|T^{-1}_{\gamma} (f\otimes g) \|_{L^{1,\infty}(L^{1,\infty})}
    \geq C_{\vec p, n} \|  f\otimes g \|_{L^{p_2,\infty}(L^{p_1,\infty})}$$
    holds, which gives
    \begin{equation}\label{eq: reverse integral}
    \|T^{-1}_{\gamma} (f\otimes g) \|_{L^1(L^{1})} \geq C_{\vec p,n} \|f\|_{L^{p_1, \infty}} \|g\|_{L^{p_2, \infty}}.
    \end{equation}
    As said in (\ref{eq: norm interpolation}), by interpolation,
    we have for all $0< p_1 < r < p_2 \leq \infty$,
    \begin{equation}\label{eq: norm interpolation again}
    \|f\|_{r} \leq C_{\vec p,r} \|f\|_{p_1, \infty}^{\theta} \|f\|_{p_2, \infty}^{1-\theta},
    \end{equation}
    where
    $$\frac{1}{r}=\frac{\theta}{p_1}+\frac{1-\theta}{p_2}.$$
    It follows from (\ref{eq: reverse integral}) that for all $0< s_1, s_2, t_1, t_2 \leq 1$ satisfying
     $1/s_1+1/t_1= 1/s_2+1/t_2 =2+\gamma/n$,
    $$\|T^{-1}_{\gamma} (f\otimes g) \|_{L^1(L^{1})}\geq C_{\gamma, n} \|f\|_{L^{s_1, \infty}} \|g\|_{L^{t_1, \infty}} $$
    and
    $$\|T^{-1}_{\gamma} (f\otimes g) \|_{L^1(L^{1})}\geq C_{\gamma, n} \|f\|_{s_2, \infty} \|g\|_{L^{t_2, \infty}}.$$
    We choose $s_1, s_2, t_1, t_2 $ such that
    \[
      \frac{1}{p_1}  =\frac{\theta}{s_1}+ \frac{1-\theta}{s_2},  \quad  0<\theta<1.
     \]
    Then
       \[
      \frac{1}{p_2}  =2+ \frac{\gamma}{n}-  \frac{1}{p_1} =  \frac{\theta}{t_1}+ \frac{1-\theta}{t_2}.
     \]
     Therefore, we see from  (\ref{eq: norm interpolation again}) that
    \begin{align}
    \|T^{-1}_{\gamma} (f\otimes g) \|_{L^1(L^{1})}
    &=\int_{\mathbb{R}^{n}}  \int_{\mathbb{R}^{n}}  f(x) g(y) (|x+y|+|x-y|)^{n(1/p_1 +1/p_2  -2)} dxdy \nonumber \\
    &\hskip -10mm\geq C_{\vec p,n} \|f\|_{L^{s_1,\infty}}^{\theta}
    \|f\|_{L^{s_2, \infty}}^{1-\theta}
    \|g\|_{L^{t_1,\infty}}^{\theta}
    \|g\|_{L^{t_2, \infty}}^{1-\theta} \nonumber  \\
    &\hskip -10mm \geq  C_{\vec p,n}  \|f\|_{L^{p_1}} \|g\|_{L^{p_2}}. \label{eq:t1}
    \end{align}

    Let $f^*$ and $g^*$ be  the symmetric decreasing rearrangements of $f$ and $g$, respectively.
   We see   from the  rearrangement inequality \cite[Theorem 3.7]{Lieb2001} that
    \begin{align*}
    &\ \ \ \int_{\mathbb{R}^{n}}  \int_{\mathbb{R}^{n}}  f(x) g(y) |x-y|^{n(1/p_1 +1/p_2  -2)} dxdy \\
    &\geq \int_{\mathbb{R}^{n}}  \int_{\mathbb{R}^{n}}  f^{\ast}(x) g^{\ast}(y)|x-y|^{n(1/p_1 +1/p_2-2)} dxdy \\
    &=\frac{1}{2} \int_{\mathbb{R}^{n}}  \int_{\mathbb{R}^{n}} f^{\ast}(x) g^{\ast}(y) \big(|x-y|^{n(1/p_1 +1/p_2-2)}+ |x+y|^{n(1/p_1 +1/p_2-2)}\big) dxdy \\
    &\geq C_{\vec p, n}\int_{\mathbb{R}^{n}}  \int_{\mathbb{R}^{n}}   f^{\ast}(x) g^{\ast}(y) \big(|x+y|+|x-y|\big)^{n(1/p_1 +1/p_2-2)} dxdy \\
    &\ge C_{\vec p,n}
      \|f^*\|_{L^{p_1}} \|g^*\|_{L^{p_2}},
    \end{align*}
where we use (\ref{eq:t1}) in the last step.
This together with the fact that
    $$\|f\|_{L^{p_1}} =\|f^{*}\|_{L^{p_1}}\quad \mathrm{ and } \quad \|g\|_{L^{p_2}} =\|g^{*}\|_{L^{p_2}} $$
    gives (\ref{eq: reverse separate fractional}).
  \end{proof}


\end{document}